\documentclass[11pt]{article}
\usepackage{amscd}
\usepackage[english]{babel}
\usepackage{amsmath}
\usepackage{amsthm}
\usepackage{amssymb}
\usepackage{enumerate}
\usepackage{color}
\newtheorem{theorem}{Theorem}[section]
\newtheorem{definition}[theorem]{Definition}
\newtheorem{lemma}[theorem]{Lemma}

\newtheorem{proposition}[theorem]{Proposition}
\newtheorem{corollary}[theorem]{Corollary}
\newtheorem{remark}[theorem]{Remark}

\newcommand{\ignore}[1]{}

\newcommand{\berg}{\mathcal{A}}
\newcommand{\besov}{\mathfrak{B}}
\newcommand{\dirichlet}{\mathcal{D}}

\title{\bf  Bloch, Besov and Dirichlet spaces of slice hyperholomorphic functions}

\author{
C. Marco Polo Castillo Villalba\\
Instituto de Matematicas \\
Cuidad Universitaria,\\
Unidad de Posgrado,\\
 C.P. 04510, M\'exico D.F. \\
mpcastillo@nemfis.mx
\and
Fabrizio Colombo\\
Politecnico di Milano\\
Dipartimento di Matematica\\Via E. Bonardi, 9\\20133 Milano,
Italy\\fabrizio.colombo@polimi.it
\and
Jonathan Gantner  \\
Vienna University of Technology\\
Institute for Analysis \\ and Scientific Computing\\
Wiedner Hauptstra\ss e 8 - 10\\
1040 Wien, Austria\\
jonathan.gantner@gmx.at
\and
J. Oscar Gonz\'alez-Cervantes\\
 Departamento de Matem\'aticas  \\
 E.S.F.M. del
I.P.N. 07338 \\
M\'exico D.F., M\'exico
\\
jogc200678@gmail.com
 }

\date{ }
\begin{document}

\maketitle
\begin{abstract}
In this paper we begin the study
of some important Banach spaces of slice hyperholomorphic functions, namely the Bloch,  Besov and weighted Bergman spaces, and we also consider the Dirichlet space, which is a Hilbert space. The importance of these spaces is well known, and thus their study in the framework of  slice hyperholomorphic functions is relevant, especially in view of the fact that this class of functions  has recently found several applications in operator theory and in Schur analysis.
We also discuss the property of invariance of this function spaces with respect to M\"obius maps by using a suitable notion of composition.
 \end{abstract}


\section{Introduction and preliminary results}

Since their introduction in 2006, see \cite{gs}, slice hyperholomorphic function have found several applications, e.g.
 in operator theory, see \cite{book_functional}, and in  Schur analysis.
Such functions are also called slice regular in the case they are defined on quaternions
and are  quaternionic-valued. In the case of functions $f :  \Omega \subseteq \mathbb{R}^{n+1}\to \mathbb{R}_n$, where
$\mathbb{R}^{n+1}$ is the Euclidean space
 and $\mathbb  R_n$ is the real Clifford algebra of real dimension $2^n$, these functions are also called slice monogenic. The main advantage of this class of functions, with respect to the more classical function theories like the one of Fueter regular functions
or of Dirac regular functions, see \cite{bds, fueter32, fueter 1, ghs}, is that it contains power series expansions of the quaternionic variable or of the paravector variable  in the case of Clifford algebras-valued functions.  In this paper we will consider the quaternionic setting.

In our recent works in Schur analysis (see  \cite{MR2002b:47144} for an overview) we have studied the Hardy space $H_2(\Omega)$ where $\Omega$ is the quaternionic unit ball $\mathbb B$ or the half space
$\mathbb H^+$ of quaternions with positive real part, \cite{acs1, acs2}.  The Hardy spaces $H^p(\mathbb B)$, $p>2$, are considered in \cite{sarfatti}. The Bergman spaces are treated in \cite{cglss, CGS, CGS3} and  for the Fock spaces see \cite{Fock}.

The aim of this paper is to continue in the systematic study of the function spaces of slice hyperholomorphic functions by deepening the study of Bergman spaces whose weighted versions are in Section \ref{WBSpaces}  of this paper. We also introduce and study some properties of the  Bloch, Besov and Dirichlet spaces on the unit ball $\mathbb B$.
In particular, we study the invariance of these spaces under the M\"obius transformations  using a suitable notion of composition of functions. In fact, as it is well known in hypercomplex analysis and thus also in the present setting, the composition of two functions satisfying a condition of holomorphicity (in this setting the slice regularity) is not, in general, holomorphic. Thus, using this notion of composition, we show the invariance of the Bloch and Besov spaces. These spaces turn our to be  Banach spaces.

To motivate the study of slice hyperholomorphic functions, we recall that from the Cauchy formulas for slice monogenic functions $f: \Omega\subset\mathbb{R}^{n+1} \to \mathbb{R}_n$ we can define the $S$-functional calculus, that  works for $n$-tuples of not necessarily commuting  operators (bounded or unbounded), see
\cite{duke, acgs, functionalcss}. This calculus can be  considered the noncommutative version of the Riesz-Dunford functional calculus.
The Cauchy formula for slice regular functions of a quaternionic variable is the main tool to define a functional calculus for quaternionic linear operators, which
 is based on the notion of $S$-spectrum for quaternionic operators,  see \cite{JGA}. This calculus
possesses most of the properties of the Riesz-Dunford functional calculus and
 allows the study of the quaternionic evolution operator, see
\cite{perturbation, evolution, GR} which is of importance in quaternionic quantum mechanics, see \cite{adler}.
It is worthwhile  to mention that the so-called slice continuous functions are the framework in which to define a continuous version of the S-functional calculus, see \cite{GMP}.

The plan of the paper is as follows. In Section \ref{BLSpaces} we treat the Bloch space  $\mathcal{B}$ on the unit ball $\mathbb B$. We recall the notion of $\circ_i$-composition and we show the invariance of $\mathcal{B}$ under M\"obius transformations with respect to this composition. We also prove some conditions on the coefficients of a converging power series belonging to the Bloch space which generalize to this setting the analogous inequalities in the case of a function belonging to the complex Bloch space. Finally,  we introduce the little Bloch space which turns out to be separable. In Section \ref{WBSpaces} we deepen the study of Bergman spaces by introducing their weighted versions.

Then, in Section \ref{BSpaces}, we move to the Besov spaces $\besov_p$. We show their invariance under M\"obius transformations, using the $\circ_i$-composition, then we introduce suitable seminorms and we study the property of being a Banach space. We also show the relation between  the Besov spaces and the weighted Bergman spaces  and $L^p$ spaces (via a Bergman type projection) and we conclude with some duality results.
Finally, in Section \ref{DSpaces} we introduce the Dirichlet space and we show that it is a Hilbert space.
\\
\\
We now recall the  main results of slice regular functions that we will need in the sequel. For more details see the book \cite{book_functional}.
We consider the space  $\mathbb R^3$ embedded in  $\mathbb H$ as follows $$(a_1,a_2,a_3) \mapsto a_1e_1+a_2e_2+a_3e_3,$$
where $\{e_0=1,e_1,e_2,e_3\}$ is the usual basis of the quaternions.
    Let $\mathbb S^2$ be the sphere of purely imaginary unit quaternions and let $i\in\mathbb S^2$. The  space generated by $\{1, i\}$, denoted by $\mathbb C(i)$,   is isomorphic, not only as a linear space but even as a field,  to  the field of the complex numbers.   Given a domain $\Omega\subset\mathbb H$, let  $\Omega_i=\Omega\cap \mathbb C(i)$  and   $Hol( \Omega_i)$ represents the  complex linear space of holomorphic functions from   $\Omega_i$  to $\mathbb C(i)$. \\
    Any nonreal quaternion $q=x_0+e_1x_1+e_2x_2+e_3x_3:=x_0+\vec q$ can be uniquely written in the form $q=x+I_qy$ where
     $x=x_0$,  $\displaystyle I_q=\frac{\vec q}{\|\vec q\|}\in\mathbb{S}^2$, and $y=\|\vec q\|$ thus it  belongs to the complex plane $\mathbb C(I_q)$.
\\
\begin{definition}[Slice regular or slice hyperholomorphic functions]
A real differentiable quaternionic-valued function $f$ defined on an open set $\Omega \subset\mathbb{H}$ is called (left) slice regular on $\Omega$ if, for any $i\in \mathbb S^2$, the function $f{\mid_{\Omega_{i}}}$
 is such that
$$
 \left( \frac{\partial}{\partial x} + i  \frac{\partial}{\partial y}\right) f\mid_{{\Omega_{i}}}(x+yi)=0, \textrm{ on $\Omega_{i}$}.
 $$
We denote by $\mathcal{SR}(\Omega)$ the set of slice regular functions on $\Omega$.
\end{definition}
\begin{lemma}[Splitting Lemma]\label{lemma spezzamento}
If $f$ is a slice regular function on an open set $\Omega$, then for
every $i \in \mathbb{S}^2$, and every $j\in\mathbb{S}^2$,
orthogonal  to $i$, there are two holomorphic functions
$F,G: \Omega\cap \mathbb{C}(i) \to \mathbb{C}(i)$ such that for any $z=x+yi$, it is
$$f\mid_{{\Omega_{i}}}(z)=F(z)+G(z)j.$$
\end{lemma}
The open sets on which  slice regular functions are naturally defined are described below.
\begin{definition}
Let $\Omega\subseteq\mathbb{H}$. We say that $\Omega$ is axially symmetric if whenever $q=x+I_qy$ belongs to $\Omega$ all the elements $x+iy$ belong to $\Omega$ for all $i\in\mathbb{S}^2$. We say that $\Omega$ is a slice domain, or s-domain for short, if it is a domain intersecting the real axis and such that $\Omega\cap \mathbb C(i)$ is connected for all $i\in\mathbb{S}^2$.
\end{definition}
A main property of slice hyperholomorphic functions is the Representation Formula, see \cite{book_functional}.

\begin{theorem}[Representation Formula]\label{formula} Let
$f$ be a slice regular function on an axially symmetric s-domain $\Omega\subseteq  \mathbb{H}$. Choose any
$j\in \mathbb{S}^2$.  Then the following equality holds for all $q=x+i y \in \Omega$:
\begin{equation}\label{relationflats}
    f(x+y i)= \frac{1}{2}(1+ i j) f(x-y j)+\frac{1}{2}(1- i  j) f(x+y j).
     \end{equation}
\end{theorem}

Let ${i,  j} \in\mathbb S^2$ be mutually orthogonal vectors  and   $\Omega\subset \mathbb H$ an axially symmetric s-domain, then the Splitting Lemma and the Representation Formula,  imply the good definition of the following operators, which  relate the slice regular space with the space  of pairs of holomorphic functions on $\Omega_i$, denoted by $Hol(\Omega_i)$. We define:
  $$
  \begin{array}{lrcl}
    Q_i : &  \mathcal{SR}(\Omega)  &\longrightarrow &  Hol(\Omega_i)+ Hol(\Omega_i)j\\
 &  & & \\
Q_i :  & f & \longmapsto & f\mid_{\Omega_i},\quad \quad  \end{array}
  $$
and
   $$ P_i :   Hol(\Omega_i)+ Hol(\Omega_i)j \longrightarrow  \mathcal{SR}(\Omega),$$
   $$
   P_i[f](q)=P_i[f](x+yI_q)=\frac{1}{2}\left[(1+ I_qi)f(x-yi) + (1- I_q i) f(x+y i)\right],
   $$
   where $f\in  Hol(\Omega_i)+ Hol(\Omega_i)j$.
 Moreover, we have  that
$$P_i\circ Q_i= \mathcal{I}_{\mathcal{SR}(\Omega)} \quad \textrm{and}
 \quad  Q_i\circ P_i= \mathcal{I}_{ Hol(\Omega_i)+ Hol(\Omega_i)j } $$
where $\mathcal{I}$ denotes the identity operator.

\section{Bloch space }\label{BLSpaces}
Let $\mathbb B = \mathbb B(0,1) \subset\mathbb H$ denote the unit ball centered at zero. Then, for each $i\in \mathbb S^2$, we denote by  $\mathbb B_i = \mathbb B \cap \mathbb C(i)$ the unit disk in the plane $\mathbb{C}(i)$.

In usual complex analysis, the Bloch space $\mathcal{B}_{\mathbb C}$ is defined to be the space of analytic functions on the unit disc $\mathbb D$ in the complex plane such that
$$
\|f\|_{\mathcal B_{\mathbb C}} = |f(0)| + \sup\{(1-|z|^2)|f'(z)|\colon z\in\mathbb D\}<\infty.
$$
Analogously, we give the following definitions.
\begin{definition}
The {\em quaternionic slice regular Bloch space} $\mathcal B$ associated with $\mathbb B$ is the quaternionic  right linear space of slice regular  functions $f$  on $\mathbb B$  such that
 $$\sup \{    ( 1-|   q|^2  ) | \frac{\partial f}{\partial x_0} (q)| \colon q\in \mathbb B\} < \infty. $$
 We define a norm on this space by
 $$ \| f\|_{\mathcal B}   =|f(0)| + \sup \{    ( 1-|   q|^2  ) | \frac{\partial f}{\partial x_0} (q)| \colon q\in \mathbb B\} < \infty. $$
 \end{definition}
 In the sequel, we will be also in need of the following definition:
 \begin{definition}\label{BICONI}
 By $\mathcal B_i$ we denote the quaternionic right linear space of slice regular functions $f$  on $\mathbb B$   such that
$$   \sup \{    (1-|z|^2)| Q_i[f]  ' (z)  |   \colon   z\in \mathbb B_i   \} <\infty. $$
We define a norm on this space by
$$  \|  f \|_{\mathcal B_i}  =|f(0) | + \sup \{    (1-|z|^2)| Q_i[f]  ' (z)  |   \colon    z\in \mathbb B_ i  \}. $$
\end{definition}
\begin{remark}{\rm
The fact that $\| f\|_{\mathcal B}$ and $\|  f \|_{\mathcal B_i}$ are norms can be verified by a direct computation.
}
\end{remark}
\begin{remark}{\rm
Note that for any $i\in\mathbb{S}^2$ the function $\displaystyle{Q}_i[f]$ is a holomorphic map of a complex variable $z = x_0 + iy$ and that for its derivative we have
$\displaystyle Q_i[f]  ' (z) = \frac{ \partial Q_i[f] }{\partial x_0} (z).$
}
\end{remark}
\begin{remark}\label{rem1}
{\rm
Let $i\in \mathbb S^2$, and let $f\in \mathcal B_i$. Then, for any $j\in \mathbb S^2$ with $j\perp i$, there exist holomorphic functions $f_1, f_2\colon \mathbb{B}_i\to\mathbb{C}(i)$ such that $Q_i[f] = f_1+  f_2j$. Moreover, as
$$
|Q_i[f]  ' (z)  |^2=  |f_1 ' (z)  |^2 +  |f_2  ' (z)  |^2,
 $$
  the condition  $f\in \mathcal B_i$  is equivalent to $f_1,f_2$ belonging to the one dimensional complex  Bloch space. As this is a Banach space one can see directly that  $(\mathcal B_i , \| \cdot \|_{\mathcal B_i})$  is also a Banach space.
}
\end{remark}

\begin{proposition}\label{prop_Bnorms}
Let  $i \in \mathbb S^2$. Then  $ f\in\mathcal B_i$ if and only if $f\in\mathcal B$. The space  $(\mathcal B , \| \cdot \|_{\mathcal B})$  is a Banach space and $(\mathcal B, \| \cdot \|_{\mathcal B})$ and    $(\mathcal B_i , \| \cdot \|_{\mathcal B_i})$ have equivalent norms. Precisely, one has
$$\|f\|_{\mathcal B_i}\leq \|f\|_{\mathcal B} \leq  2\|f\|_{\mathcal B_i}. $$
\end{proposition}
\begin{proof}
 As  $\mathbb B_i \subset\mathbb B$, for any $f\in \mathcal{B}$,  one has $ \|  f \|_{\mathcal B_i}  \leq \|f\|_{\mathcal B} $ by definition. Therefore  $ \mathcal B \subset\mathcal B_i$. \\
 On the other hand, let $f\in \mathcal B_i$. Then $\frac{\partial f}{\partial x_0}$ is slice regular. Thus, we can apply the Representation Formula for slice regular functions and we obtain
$$  \frac{\partial f}{\partial x_0}(q) = \frac{1}{2} [   (1-Ii)\frac{\partial f}{\partial x_0} (z) + (1+Ii) \frac{\partial f}{\partial x_0}(\bar  z)    ] , $$
where $q=x+Iy$ and $z=x+iy$. Applying the triangle inequality, as $|1-Ii|\leq 2$,  $|1+Ii|\leq 2$, we obtain
$$
|\frac{\partial f}{\partial x_0}(q)| \leq |\frac{\partial f}{\partial x_0} (z) |+  | \frac{\partial f}{\partial x_0}(\bar  z)    |  ,
$$
and as $|q| = |z| = |\bar{z}|$ we get
 $$(1-|q|^2) |\frac{\partial f}{\partial x_0}(q)| \leq (1-|z|^2)  |\frac{\partial f}{\partial x_0} (z) |+ (1-|\bar z|^2) | \frac{\partial f}{\partial x_0}(\bar  z)    | \leq 2\| f\|_{\mathcal B_i} .$$
Taking the supremum over all $q\in\mathbb{B}$, one concludes that
$$\|f\|_{\mathcal B} \leq 2\|f\|_{\mathcal B_i}.$$
Thus, $f\in\mathcal{B}$. Furthermore, the norms $\| \cdot\|_{\mathcal{B}_i}$ and $\|\cdot\|_{\mathcal{B}}$ are equivalent and $\mathcal{B}$ turns out to be a Banach space.

\end{proof}
\begin{remark}{\rm  For  $z=x+iy \in \mathbb B_i$ and $w=x+jy \in \mathbb B_j$, from the proof of the previous result, in particular, we have that
$$(1-|z|^2) | Q_{i}[f]'(z)| \leq  (1-|w|^2) | Q_{j}[f]'(w)| + (1-|\bar w|^2) | Q_{j}[f]'(\bar w)| \leq 2 \|f\|_{\mathcal B_j}.
$$
Taking the supremum over all $z\in\mathbb B_i$ we get
$$\|f\|_{\mathcal B_i} \leq 2\|f\|_{\mathcal B_j}. $$
Thus, we have that  $f\in\mathcal B_i$ if and only if $f\in \mathcal{B}_j$ and that the norms $\|\cdot\|_{\mathcal{B}_i}$ and $\|\cdot\|_{\mathcal{B}_j}$ are equivalent.
}
\end{remark}

\begin{remark}{\rm
In usual complex analysis  it´s well known that the $\mathcal B_i$  is not separable, therefore $\mathcal B$ is also not separable. In fact, following the complex case, see \cite{danikas},  one can consider the family of  slice regular functions   on  $\mathcal B$
$$  f_n(q) = P_i\Big[ \frac{e^{-in}}{2}  \log  \Big(\frac{  1+ e^{-in}  z }{1-e^{-in} z  }\Big)\Big] (q),\ \ \ \ n\in \mathbb{N}
$$
 where $q=x+Iy$ and  $z=x+iy$,  which has the following property:
$$   \| f_n - f_m \|_{\mathcal B} \geq 1, $$
so $\mathcal B$ is not separable.
}
\end{remark}

Similar to the usual complex case, see \cite{kehe}, we formulate the following definition, which can be found also in \cite{sarfatti}.
\begin{definition}
We define the space $\mathcal H^{\infty}  (\mathbb B) $ as the quaternionic right linear space of bounded slice regular functions $f$ defined on $\mathbb B$, i.e. of all slice regular functions such that
$$   \|f\|_{\infty } = \sup \{  |f(q)| \ \colon \ q\in \mathbb B\} < \infty.$$
\end{definition}
We will be also in need of the following:
\begin{definition}
We denote by $\mathcal H_i^{\infty}  (\mathbb B)$  the quaternionic right linear space of slice regular functions  $f$ on $ \mathbb B$ such that
$$   \|f\|_{\infty, i  } =  \sup \{  |Q_{i}[f](z)| \colon z\in \mathbb B_i\} <\infty.$$
\end{definition}

\begin{remark}{\rm
Let $i\in \mathbb S^2$ and let $f\in    \mathcal H_i^{\infty}  (\mathbb B_i)$. Then for any $j\in\mathbb{S}^2$ with $j\perp i$ there exist holomorphic functions $f_1,f_2\colon \mathbb{B}_i\to\mathbb{C}(i)$ such that $Q_i[f] = f_1+f_2 j$. It is easy to show that $f\in    \mathcal H_i^{\infty}  (\mathbb B_i)$ if and only if $f_1$ and $f_2$ belong to the usual complex space $\mathcal H_{\mathbb{C}}^{\infty}(\mathbb B_i)$ of bounded holomorphic functions defined in $\mathbb B_i$. Moreover,  as  $\mathcal H_{\mathbb{C}}^{\infty}(\mathbb B_i)$ is a Banach space, one can see that $   \mathcal H_i^{\infty}  (\mathbb B_i)$ is also a Banach space.
}
 \end{remark}

\begin{proposition}
Let $i\in \mathbb S^2$. Then $f\in\mathcal H^{\infty} _i  (\mathbb B)$ if and only if $f\in \mathcal H^{\infty}  (\mathbb B) $. The space $  (  \mathcal H^{\infty}  (\mathbb B) ,   \|\cdot\|_{\infty } )$ is a Banach space. Moreover, the norms $\|\cdot\|_{\infty}$ and $\|\cdot\|_{\infty,i}$ are equivalent. Precisely, one has
$$  \|f\|_{\infty, i}\leq \|f\|_{\infty }  \leq  2\|f\|_{\infty, i}.$$
\end{proposition}
\begin{proof}
By definition, we have $\| f\|_{\infty,i}\leq \| f\|_{\infty}$.  Furthermore, from the Representation Formula it follows that  for any $f\in \mathcal{SR}(\Omega)$ one has
$$   \|f\|_{\infty }  \leq  2\|f\|_{\infty, i}.$$
Therefore, $\| f\|_{\infty} < \infty$ if and only $\| f\|_{\infty,i} <\infty$, the norms are equivalent and, as $\mathcal{H}_i^{\infty}(\mathbb{B})$ is a Banach space, $\mathcal{H}^{\infty}(\mathbb{B})$ is a Banach space too.

\end{proof}

\begin{proposition} Let $f\in\mathcal H^{\infty}(\mathbb B)$. Then $f\in \mathcal B$ and $\| f\|_{\mathcal B} \leq 4\|f\|_{\infty}.$
\end{proposition}
\begin{proof} Let $f\in\mathcal H^{\infty}(\mathbb B)$ and $i,j\in\mathbb{S}^2$ such that $i\perp j$. Then there exist holomorphic functions $f_1,f_2\colon \mathbb{B}_i\to\mathbb{C}(i)$ such that $\mathcal{Q}_i[f] = f_1 + f_2j$.
 From usual complex analysis, it is well known that, for any holomorphic function $g$ defined on the complex unit disc, the estimates $\| g \|_{\mathcal{B}_ \mathbb{C}} \leq \| g\|_{\infty}$ holds (see e.g. \cite{kehe}).Thus one has
 \begin{align*}
 \|f\|_{\mathcal{B}_i}&\leq |f_1(0)| + \sup_{z\in\mathbb{B}_i}(1-|z|^2)|f_1'(z)| + |f_2(0)|+\sup_{z\in\mathbb{B}_i}(1-|z|^2)|f_2'(z)|  \\
&=\|f_1\|_{\mathcal{B}_\mathbb{C}} + \|f_2\|_{\mathcal{B}_\mathbb{C}} \leq\|f_1\|_{\infty} + \|f_2\|_{\infty} \leq 2\| f\|_{\infty,i}.
 \end{align*}
Therefore, from the previous propositions,  one obtains that  $$\|f\|_{\mathcal B}\leq 2 \|f\|_{\mathcal{B}_i} \leq 4\|f\|_{\infty,i} \leq 4  \|f\|_{\infty}.$$
In particular, $\| f\|_{\mathcal B} < \infty$ if  $\| f\|_{\infty}<\infty$.
\end{proof}

In general, as in the complex case (see \cite{kehe}), $f\in\mathcal B$ does not imply  $f\in\mathcal{H}^{\infty}(\mathbb{B})$. For example the function
$$ f(q)= P_{i} [\log(1-z)](q),   $$
belongs to $\mathcal B$, but $Q_{i}[f]$ does not belong  to $ \mathcal H_i^{\infty}  (\mathbb B_i) $ and so  $f\notin  \mathcal H^{\infty}  (\mathbb B) $.
\\

The complex Bloch space is important because of its invariance with
respect to M\"obius transformation. As it is well known, when dealing
with hyperholomorphic functions
(not only in the slice regular setting but also in more classical
settings, like in the Cauchy-Fueter regular setting) the composition of
hyperholomorphic
functions does not give, in general, a function of the same type.
The M\"obius transformation we will consider are those which are slice
regular, see \cite{caterina}, and we will define a suitable notion of composition
which will allow us to prove invariance under M\"obius transformation.\\
For the  $*$-product of slice hyperholomorphic functions used below,  we refer the reader to  \cite{book_functional}.

\begin{definition}\label{def_moebius}
For any $a\in \mathbb B\setminus \mathbb R$ we define the slice regular M\"obius transformation as
$$T_a(q)=(1- q\bar a ) ^{-\ast}\ast (a-q), \quad q\in\mathbb B,$$
where $\ast$ denotes the slice regular product.
 \end{definition}

\begin{proposition}\label{Moeb_prop} Let $a\in\mathbb{B}\setminus\mathbb{R}$ and $\displaystyle I=\frac{\vec{a}}{\|\vec{a}\|}$, or
 $a\in\mathbb{R}$ and $I$ is any element in $\mathbb{S}^2$. Then the slice regular M\"obius transformation $T_a$ has the following properties:
\begin{enumerate}[(i)]
\item $T_{a}$ is a bijective mapping of $\mathbb B$ onto itself.
\item \label{Moeb_red}$ \displaystyle T_a(z)= \frac{a-z}{1-\bar{a} z} $   for all $ z\in   \mathbb B_{I}$, i.e.  on $\mathbb B_I$ the slice regular  M\"obius transformation coincides with the usual one dimensional complex  M\"obius transformation.
\item $T_a(0) = a$, $T_a(a) = 0$ and $T_a\circ T_a(q) = q$ for all $q\in \mathbb{C}(I)$.
\end{enumerate}
\end{proposition}

We note that the function $T_a$ is such that $T_a:\  \mathbb
B\cap\mathbb C_I\to \mathbb C_I$. Thus, using Proposition 2.9 in \cite{CGS},
we can give the following definition.

\begin{definition}
 Let $a\in\mathbb{B}\setminus\mathbb{R}$ and $\displaystyle I=\frac{\vec{a}}{\|\vec{a}\|}$, or
 $a\in\mathbb{R}$ and $I$ is any element in $\mathbb{S}^2$.
Let $T_a$ be the associated M\"obius transformation. For any $f\in \mathcal{SR}(\mathbb B) $ we define the function $f\circ_I T_a\in\mathcal {SR}(\mathbb B)$ as
$$      f\circ_IT_a:= P_I[ (  \  f_1    \    \circ \ {Q}_I[T_a]   \ ) ]   +  P_I[ (  \  f_2    \    \circ \  Q_I[T_a]    \ ) ]     j ,$$
where $ j \in \mathbb S^2$ with $j\perp I$ and $f_1,f_2\colon \mathbb B_I\to\mathbb C(I) $ are holomorphic functions such that
${Q}_I[f] = f_1+f_2j$.
\end{definition}

Using this notion of composition, we can prove a result on invariance under M\"obius transformation.

\begin{proposition} Let $a\in\mathbb{B}\setminus\mathbb{R}$ and $\displaystyle I=\frac{\vec{a}}{\|\vec{a}\|}$, or
 $a\in\mathbb{R}$ and $I$ is any element in $\mathbb{S}^2$.
Then for any $f\in \mathcal B$, one has  $ f\circ_I T_a \in \mathcal B $.
\end{proposition}
\begin{proof}
Let $f\in \mathcal B$ and let $I = \frac{\vec{a}}{\|\vec{a}\|}$, if $a\in \mathbb{R}$
then in the computations below we can use any $I\in \mathbb{S}$.  For $q= x_0 +i y \in \mathbb B$  let $z = x_0 + I y$. As  $g:=\frac{\partial}{\partial x_0}( f\circ_I T_a)$ is slice regular on $\mathbb B$, we can apply the Representation Formula. Thus, as $|q| = |z| = |\bar{z}|$, we have
\begin{align*}
 (1-|q|^2&)^2|g(q)|^2 =   (1-|q|^2)^2|\frac{1}{2}(1-iI) g(z) + \frac{1}{2}(1+iI)g(\bar{z})|^2
 \\
&= (1-|q|^2)^2\frac{1}{4}\left(|1-iI|^2|g(z)|^2 + 2\mathrm{Re}\left((1-iI)g(z)\overline{g(\bar{z})}(\overline{1+iI})\right)
+ | 1 + iI|^2|g(\bar{z})|^2 \right)
\\
&= \frac{1}{4}\left(|1-iI|^2(1-|z|^2)^2|g(z)|^2 + 2\mathrm{Re}\left((1-iI)(1-|z|^2)g(z)\overline{(1-|\bar{z}|^2)}\, \overline{g(\bar{z})}(\overline{1+iI})\right) \right.
\\
&\phantom{= \frac{1}{4}}\left.+ | 1 + iI|^2(1-|\bar{z}|^2)^2|g(\bar{z})|^2\right).
\end{align*}
Now let us denote
\begin{equation*}
  f'|_{\mathbb{C}_I}= \frac{\partial}{\partial x_0} Q_I[f] \qquad \text{and} \qquad T_a'|_{\mathbb{C}_I} = \frac{\partial}{\partial x_0}  Q_I[T_a].
\end{equation*}
Observe that $f'|_{\mathbb{C}_I}$ and $ T_a'|_{\mathbb{C}_I}$, defined above, depend on $I\in \mathbb{S}^2$, but we omit the subscript in the rest of the proof.
Then we have $g(z) = f'(T_a(z))T_a'(z)$ as $z\in\mathbb{C}(I)$. If we set $w = T_a(z)$,  we know that
\begin{equation*}
|1 - |w|^2| = (1-|z|^2)|T_a'(z)|
\end{equation*}
as the function ${Q}_I[T_a]$ is nothing but the complex M\"obius transformation on $\mathbb{B}_I$. Thus, if we put $w^{\ast} = T_a(\bar{z})$, we have
\begin{align*}
(1-|z|^2)|g(z)| &= (1-|z|^2)|f(T_a(z))||T_a'(z)| =  (1-|w|^2)|f'(w)| \\
(1-|\bar{z}|^2)|g(\bar{z})| &= (1-|\bar{z}|^2)|f(T_a(\bar{z}))||T_a'(\bar{z})| =  (1-|w^{\ast}|^2)|f'(w^{\ast})|.\\
\end{align*}
Plugging this into the above equation, we obtain
\begin{align*}
 (1-|q|^2)^2|g(q)|^2 &= \frac{1}{4}\Big[|1-iI|^2(1-|w|^2)^2|f'(w)|^2
 \\
 &+2\mathrm{Re}\left((1-Ii)(1-|w|^2)f'(w)\overline{(1-|w^{\ast}|^2)}\, \overline{g(w^{\ast})}(\overline{1+Ii})\right)
 \\
&\phantom{= \frac{1}{4}}+ | 1 + Ii|^2(1-|w^{\ast}|^2)^2|f'(w^{\ast})|^2\Big]\\
&= \frac{1}{4}\left|((1-iI)(1-|w|^2)f'(w)+ (1 + Ii)(1-|w^{\ast}|^2)f(w^{\ast})\right|^2\\
&\leq ((1-|w|^2)|f'(w)| + |(1-|w^{\ast}|^2)|f'(w^{\ast})|)^2.
\end{align*}
Recalling the definitions of $g$ and $f'$ we get
\begin{align*}
(1-|q|^2)\left|\frac{\partial}{\partial x_0}(f\circ_I T_a)(q) \right| &\leq (1-|w|^2)|f'(w)| + |(1-|w^{\ast}|^2)|f'(w^{\ast})| \\
&\leq 2\sup\left\{(1-|x|^2)|\frac{\partial f}{\partial x_0}(x)| \colon x\in\mathbb{B}_I\right\} <\infty.
\end{align*}
Therefore, $\| f\circ_I T_a\|_{\mathcal B_I}$ is bounded, that is $f\circ_I T_a\in\mathcal B_I$ and so $f\circ_I T_a\in\mathcal B$.

\end{proof}

The condition on a function in order to belong to the Bloch space
involves a first derivative. Here we show that it is equivalent to a
condition on the $n$-th derivative.

\begin{proposition}
For any $f\in \mathcal B$ and any $n\in \mathbb N$ with $n\geq 2$ the following inequality holds:
$$\sup \{   (1-|q|^2) ^n | \partial^n_{x_0} f (q)| \ \mid \ q \in \mathbb B \} \leq 2^{2n+2}(n-1)! \|f\|_{\mathcal B} .$$
Conversely,  if  $ f\in \mathcal {SR}(\mathbb B) $  for any $n\in \mathbb N$ with $n\geq 2 $ one has
$$
\sup \{   (1-|q|^2) ^n | \partial^n_{x_0} f (q)| \colon q \in \mathbb B \} < \infty
$$
then $f\in \mathcal B$.
\end{proposition}

\begin{proof}
Let  $f\in \mathcal B$ and $i, j\in \mathbb S^2$  with $j\perp i $. Then there exist holomorphic functions $f_1, f_2\colon\mathbb B_i\to\mathbb C(i)$ such that $ Q_i[f] = f_1+f_2 j$.
From Remark  \ref{rem1}, it follows that $f_1$ and $f_2$ lie in the complex Bloch space.
Due to Theorem 1 in \cite{danikas}, we have
\begin{align*}
\sup\{(1-|q|^2)^n |\partial_{x_0}^n f (q)& | \colon q\in \mathbb B_i\} \leq    \sup\{(1-|z|^2)^n |\partial_{x_0}^n f_1 (z) | \colon z\in \mathbb B_i\}
\\
&
+   \sup\{(1-|z|^2)^n |\partial_{x_0}^n f_2(z) | \colon z\in \mathbb B_i\}
\\
&\leq    (n-1)!2^{2n} \|f_1\|_{\mathcal B_{\mathbb C}}
+      (n-1)!2^{2n} \|f_2\|_{\mathcal B_{\mathbb C}}
 \\
 &
 \leq  (n-1)!2^{2n+1} \|f\|_{\mathcal B_i}
 \\
 &\leq  (n-1)!2^{2n+1} \|f\|_{\mathcal B}.
\end{align*}
Now let $z = x+Iy$ where $I\in\mathbb S^2$ and $q = x + i y$. As $|1+Ii|\leq 2$, $|1-Ii| \leq 2$, by applying the Representation Formula, one obtains
\begin{align*}
(1-|z|^2)^{n}|\partial_{x_0}^nf(z)|& =  (1-|q|^2)^{n}\big(\left|(1-Ii)\partial_{x_0}^nf(q) + (1+Ii)\partial_{x_0}^nf(\bar{q})\right|\big)
\\
&
\leq(1-|q|^2)^{n}|\partial_{x_0}^nf(q)| + (1-|\bar{q}|^2)^{n}|\partial_{x_0}^nf(\bar{q})|
\\
&
\leq 2\sup\{(1-|q|^2)^n |\partial_{x_0}^n f (q) | \colon q\in \mathbb B_i\}.
\end{align*}
Thus, one has
\begin{align*}
\sup\{(1-|z|^2)^n |\partial_{x_0}^n f (z) | \colon z\in \mathbb B\} &\leq 2\sup\{(1-|q|^2)^n |\partial_{x_0}^n f (q) | \colon q\in \mathbb B_i\}
\\
&
 \leq 2 (n-1)!2^{2n+1} \|f\|_{\mathcal B}.
\end{align*}
Conversely, if
$$
\sup \{   (1-|q|^2) ^n | \partial^n_{x_0} f (q)| \colon q \in \mathbb B \} < \infty,
$$
one has
$$     \sup \{   (1-|q|^2) ^n | \partial^n_{x_0} f_k (q)| \colon \ q \in \mathbb B \}  < \infty, \quad k=1,2$$ as $|\partial_{x_0}^nf_k(z)| \leq |\partial_{x_0}^nf(z)|, z\in\mathbb B$. But then, from  Theorem 1 in \cite{danikas}, it follows that $f_1,f_2$ belong to the complex Bloch space $\mathcal B_{\mathbb C}$ and so $f\in \mathcal B$.

\end{proof}

The condition on the derivatives studied in the above proposition translates
into a condition on the coefficients of the series expansion of a function in $\mathcal B$.

\begin{proposition} Let $f\in \mathcal B$ and let $(a_n)_{n\in\mathbb{N}} \subset\mathbb H$ such that
$$ f(q) = \sum_{n=0}^\infty q^n a_n.$$
Then $$ |a_n| \leq \frac{e}{\sqrt{2}} \|f\|_{\mathcal B},$$
for any $n\in\mathbb N\cup\{0\}$.
\end{proposition}
\begin{proof}
Let $i,j\in\mathbb S^2$ such that $i\perp j$, let $z\in\mathbb C(i)$ and let $f_1,f_2\colon \mathbb B_i\to \mathbb C(i)$ such that $f = f_1 + f_2 j$. Furthermore, for any $n\in\mathbb N$, let $\alpha_n, \beta_n \in  \mathbb C(i)$ such that $a_n = \alpha_{n} + \beta_{n}j$.
Then we have
$$f(z) = \sum_{n=0}^\infty z^n \alpha_n +   \sum_{n=0}^\infty z^n\beta_n j= f_1(z) + f_2 (z) j.$$
From Remark \ref{rem1} it follows that $f_1,f_2$ lie in the complex Bloch space on $\mathbb{B}_i$. Therefore, from Theorem 2, \cite{danikas} ,  it follows that for any $n\in\mathbb N$ one has
$$ |\alpha_n| \leq   \frac{e}{2} \|f_1\|_{\mathcal  B_{\mathbb{C}}}, \quad |\beta_n| \leq   \frac{e}{2} \|f_2\|_{\mathcal{B}_{\mathbb{C}}}$$
and so, as $\|f_k\|_{\mathcal{B}_{\mathbb{C}}} \leq \|f\|_{\mathcal{B}_{i}}, k=1,2$, one obtains
$$
|a_n|^2 = |\alpha_n|^2 + |\beta_n|^2 \leq \frac{e^2}{4} (\|f_1\|_{\mathcal  B_{\mathbb{C}}} ^2 + \|f_2\|_{\mathcal  B_{\mathbb{C}}}^2) \leq \frac{e^2}{2}\|f\|_{\mathcal{B}_{i}}^2 .
$$
Thus, by Proposition \ref{prop_Bnorms}, one concludes that
\begin{equation*}
\|a_n\| \leq \frac{e}{\sqrt{2}}\|f\|_{\mathcal B_i} \leq \frac{e}{\sqrt{2}}\| f\|_{\mathcal B}.
\end{equation*}

\end{proof}

Another property of the slice regular Bloch space is the following:
\begin{proposition}Let $f\in \mathcal {SR}(\mathbb B)$, let $(n_k)_{k\in\mathbb{N}}\subset\mathbb N\cup \{0\}$ and   $(a_{n_k})_{k\in\mathbb N} \subset\mathbb H$  be a sequence of quaternions such that
$$ f(q)= \sum_{k=0}^\infty q^{n_k} a_{n_k}.$$
If there exist constants $\alpha> 1$  and $M>0$ such that  and
\begin{equation}\label{equ2} \frac{n_{k+1}}{n_k} \geq \alpha, \quad   \quad  |a_{n_{k}} |\leq M, \quad  \quad \forall  k\in \mathbb N,\end{equation}
then $f\in \mathcal B$.
\end{proposition}
\begin{proof}Let $i, j\in\mathbb S^{2}$ such that $i\perp j$ and let $a_{n_k} = \alpha_{n_k} + \beta_{n_k}j$, for $k \in\mathbb{N}$. Then $$Q_i[f](z)= \sum_{k=0}^\infty z^{n_k} \alpha_{n_k}   +    \sum_{k=0}^\infty z^{n_k} \beta_{n_k} j. $$
Note that the coefficients $\alpha_{n_k}$ and $\beta_{n_k}$ satisfy \eqref{equ2}. Thus by Theorem 4 in \cite{danikas} it follows, that the functions
$$
f_1(z) =  \sum_{k=0}^\infty z^{n_k} \alpha_{n_k}  , \quad  \quad f_2(z) =   \sum_{k=0}^\infty z^{n_k} \beta_{n_k},
$$
belong to the complex Bloch space $\mathcal B_{\mathbb C}$. Thus, $f\in\mathcal B_i$ and $f\in\mathcal{B}$ follows from Remark \ref{rem1} and Proposition \ref{prop_Bnorms}.

\end{proof}

As we pointed out before, the Bloch space $\mathcal B$ is not separable. In usual complex analysis, the little Bloch space $\mathcal{B}_{\mathbb{C}}^0$ of all functions $f\in \mathcal{B}_{\mathbb{C}}$ such that $\lim_{|z|\nearrow1}(1-|z|^2)f'(z) = 0$ is a separable subspace of
$\mathcal{B}_{\mathbb{C}}$ which is of interest on its own. Thus we give the following  definition.
\begin{definition}
The slice regular little Bloch space $\mathcal{B}^0$ is the space of all functions $f\in\mathcal B$ such that
 $$\lim_{|q|\nearrow1} (1-|q|^2) |\partial_{x_0} f(q)| =0 .$$
\end{definition}
\begin{remark}\label{remB0}
{\rm
Let $f\in \mathcal {SR}(\mathbb B) $ and  $i,j\in \mathbb S^2$ with $i\perp j$ and let $f_1, f_2\colon \mathbb B_i\to \mathbb C(i)$ be holomorphic functions such that $ Q_i[f] = f_1+f_2j$. Then one has
$$ |\partial_{x_0} f(z)|^2 =  |\partial_{x_0} f_1(z)|^2 +  |\partial_{x_0} f_2(z)|^2, \quad \forall z\in \mathbb B_i .$$
Due to the Representation Formula, this implies that     $f\in \mathcal{B}^0$  if and only if $f_1, \ f_2$ belong to the complex little Bloch space $\mathcal{B}_{\mathbb{C}}^0$.
}
\end{remark}

\begin{proposition}
The little Bloch space $\mathcal{B}^0$ is the closure with respect to $\|\cdot\|_{\mathcal B}$ of the set of quaternionic polynomials of the form
$$ P(q)= \sum_{k=0}^{n}q^ka_k,\ \ \ \ a_k\in \mathbb H.
$$
In particular $\mathcal{B}^0$ is separable.
\end{proposition}
\begin{proof}
Let $f\in\mathcal{B}^0$ and $i,j\in\mathbb{S}$ such that $i\perp j$. Furthermore, let $f_1, f_2\colon \mathbb B_i\to \mathbb C(i)$ be holomorphic functions such that $ Q_i[f] = f_1+f_2j$. As the set of complex polynomials is dense in the complex little Bloch space (see Corollary 5.10 in \cite{kehe}), by the above remark, there exist complex polynomials $p_{1,n}(z) = \sum_{k=0}^{n}z^k\alpha_{n,k}$ and $p_{2,n}(z) = \sum_{k=0}^{n} z^k\beta_{n,k}$ such that $\| f_k - p_{k,n}\|_{\mathcal B_{\mathbb C}} \to 0$ as $n\to\infty$ for $k=1,2$.
Let $a_{n,k} = \alpha_{n,k} + \beta_{n,k}j$ and $p_{n}(q) = \sum_{k=0}^n q^ka_{n,k}$. Then one has
\begin{align*}
\| p_n - f\|_{\mathcal B} \leq
 2\| p_n - f\|_{\mathcal B_{i}} &\leq   2(\|p_{1,n} - f_1\|_{\mathcal{B}_{\mathbb{C}}}+ \|p_{2,n} - f_2\|_{\mathcal{B}_{\mathbb{C}}}) \overset{n\to\infty}{\longrightarrow} 0.
\end{align*}
\end{proof}

 \begin{proposition} Let $f\in \mathcal B$, then  $f\in \mathcal B^0$ if and only if   one has that
 \begin{equation}
 \label{RcondH}
 \lim_{r \nearrow 1}  \|f_r - f\|_{\mathcal{B}}=0,
 \end{equation}
 where $f_r(q) = f(rq)$ for all $q\in\mathbb B$ and $r\in(0,1)$.
\end{proposition}
\begin{proof}  Let $f\in \mathcal  B$ and  $i,j\in \mathbb S^2$ with $i\perp j$ and let $f_1, f_2\colon \mathbb B_i\to \mathbb C(i)$ holomorphic functions such that $ Q_i[f] = f_1+f_2j$. Then, for $k=1,2$, one has
 \begin{equation*}
\|f_{k,r} - f_k\|_{\mathcal{B}_{\mathbb C}  }\leq \|f_r - f\|_{\mathcal{B}_i}\leq \|f_r - f\|_{\mathcal{B}} \leq 2\|f_r - f\|_{\mathcal{B}_i} \leq
2\|f_{1,r} - f_1\|_{\mathcal{B}_{\mathbb C}}+2\|f_{2,r} - f_2\|_{\mathcal{B}_{\mathbb C}}.
 \end{equation*}
Thus, \eqref{RcondH} is satisfied if and only if
\begin{equation*}\label{RcondC}
\lim_{r \nearrow 1}  \|f_{k,r} - f_k\|_{\mathcal{B}_{\mathbb C}}=0,\qquad k = 1,2,
\end{equation*}
where $f_{k,r}(z) = f_k(rz), k=1,2$ for all $z\in\mathbb B_i$ and $r\in(0,1)$. But this is equivalent to $f_1,f_2 \in\mathcal B_{\mathbb C}^0$, see Theorem 5.9 in \cite{kehe}. Thus, from Remark \ref{remB0}, it follows that $f\in \mathcal B^0$ if and only if \eqref{RcondH} holds.

\end{proof}

We conclude this section with a result that holds on a slice.

\begin{proposition} Let $f\in \mathcal B$ and let $i\in \mathbb{S}^2$ be fixed. Then for all $q,u\in \mathbb B_i$, one has
$$  |f (q)  - f(u)| \leq \|f\|_{\mathcal B}\, d(q,u), $$
where   $$d(q,u )= \frac{1}{2} \log\left(    \frac{  1  +   \frac{| q-u  |}{ |1- \bar q u|}           }{      1  -   \frac{| q-u  |}{ |1- \bar q u|}          } \right) .$$
\end{proposition}
\begin{proof}
If $q$ and $u$ lie on the same complex plane, i.e.  $ q,u \in \mathbb B_i$ for some $i\in\mathbb S^2$,  then again for $j\in\mathbb S$ with $j\perp i$ there exist holomorphic functions $f_1,f_2\colon \mathbb B_i\to \mathbb C(i)$ such that $\mathcal{Q}_i[f] = f_1 + f_2j$. Thus, as $\|f_k\|_{\mathcal{B}_{\mathbb{C}}} \leq \|f\|_{\mathcal{B}_{i}}, k=1,2$ we have
\begin{align*}
|f(q) - f(u)|^2 &=|f_1(q) - f_1(u)|^2  + |f_2(q) - f_2(u)|^2
\\
&
\leq
\|f_1\|_{\mathcal B_{\mathbb{C}}}^2  d(q,u)^2+ \|f_2\|_{\mathcal B_{\mathbb{C}}}^2  d(q,u)^2
 \\
&
\leq 2 \|f\|_{\mathcal B_{i}}^2   d(q,u)^2
\\
&
\leq  2 \|f\|_{\mathcal B}^2   d(q,u)^2
\end{align*}
in which we have used the Theorem 3 in \cite{danikas}.
\end{proof}

\section{Weighted Bergman spaces}\label{WBSpaces}

Bergman spaces in the slice hyperholomorphic setting have been studied in  \cite{cglss,CGS,CGS3}. Here we deepen this study and we introduce weighted slice regular Bergman spaces.\\
Let $dA$ be the area measure on the unit ball of the complex plane $\mathbb D$, normalized
so that the area 0f $\mathbb D$ is $1$, i.e. $dA = \frac{1}{\pi}dxdy$. For $\alpha>-1$ let $dA_{\alpha} = (\alpha+1)(1-|z|^2)^{\alpha}dA(z)$.

Then the complex Bergman space $\berg_{\alpha,\mathbb C}^{p}$  is defined as the space of all holomorphic functions $f$ on $\mathbb D$ such that $f\in L^p(\mathbb D, dA_{\alpha})$. With $\|\cdot\|_{p,\alpha}$ we will note the $L^p$-norm on $\mathbb D$ with respect to $dA_{\alpha}$.
We begin with the following definitions.
\begin{definition}\label{4.1bis}
For $i\in\mathbb S^2$ let $dA_i$ be the normalized differential of area in the plane $\mathbb C(i)$ such that the area of $\mathbb B_i$ is equal to  $1$. Moreover, for $\alpha>-1$ let
\begin{equation*}
dA_{\alpha,i}(z) = (\alpha + 1)(1-|z|^2)^{\alpha}dA_i(z).
\end{equation*}
Then, for $ p >0$,   we define the {\em weighted slice regular Bergman space $\berg_{\alpha}^{p}$} as
the quaternionic right linear space of all slice regular functions on $\mathbb{B}$ such that
\begin{equation*}
\sup_{i\in\mathbb S^2}\int_{\mathbb B_i}|f(z)|^pdA_{\alpha,i}(z) < \infty.
\end{equation*}
Furthermore, for each function $f\in\berg_{\alpha}^p$ we define
$$ \| f \|_{p,\alpha} = \sup_{i\in\mathbb S^2}\left(\int_{\mathbb B_i}|f(z)|^p dA_{\alpha,i}(z)\right)^{\frac{1}{p}}.$$
\end{definition}
\begin{definition}\label{4.1}
For $i\in\mathbb S^2$, $ p >0$,   we define  $\berg_{\alpha,i}^{p}$ as
the quaternionic right linear space of all slice regular functions on $\mathbb{B}$ such that
\begin{equation*}
\int_{\mathbb B_i}|f(z)|^pdA_{\alpha,i}(z) < \infty,
\end{equation*}
that is $\berg_{\alpha,i}^{p} = \mathcal{SR}(\mathbb B)\cap L^p(\mathbb B, dA_{\alpha,i})$.
Furthermore, for each function $f\in\berg_{\alpha,i}^p$ we define
$$ \| f \|_{p,\alpha,i} = \left(\int_{\mathbb B_i}|f(z)|^p dA_{\alpha,i}(z)\right)^{\frac{1}{p}}.$$
\end{definition}
As we shall see  in the next section, the weighted slice regular Bergman spaces are related with the Besov spaces.
\begin{remark}\label{bergmanf1f2}
{\rm
Let $j\in\mathbb S^2$ be such that $j\perp i$. Then there exist holomorphic functions $f_1,f_2\colon\mathbb B_i\to\mathbb C(i)$ such that
$ Q_i[f] = f_1 + f_2j$. Furthermore, for all $z\in\mathbb B_i$, we have
\begin{equation}\label{eq_bergmanf1f2}| f_k(z)|^p\leq| f(z)|^p \leq 2^{\max\{0,p-1\}}\left(|f_1(z)|^p+ | f_2(z)|^p\right)\qquad k=1,2.\end{equation}
Here, the first inequality is trivial.  For $p>1$, the second one follows from the convexity of the function $x\mapsto x^p$ on $\mathbb R^+$, in fact

\begin{align*}
|f(z)|^p &\leq 2^p\left(\frac{1}{2}|f_1(z)| + \frac{1}{2}|f_2(z)|\right)^p
\\
& \leq2^p\left(\frac{1}{2}|f_1(z)|^p + \frac{1}{2}|f_2(z)|^p\right)
\\
&
 = 2^{p-1}\left(|f_1(z)|^p+ | f_2(z)|^p\right).
\end{align*}

On the other hand, for $0<p\leq1$ and $a,b\geq0$, the inequality $(a+b)^p\leq a^p + b^p$ holds. Therefore, in this case, we have
$$ |f(z)|^p \leq (|f_1(z)| + |f_2(z)|)^p \leq |f_1(z)|^p + |f_2(z)|^p.$$
Note that the inequality  \eqref{eq_bergmanf1f2} implies that $f$ is in $\berg^p_{\alpha,i}$ if and only if $f_1$ and $f_2$ are in $\berg^p_{\alpha, \mathbb C}$.
}
\end{remark}
Moreover we have:

\begin{proposition}\label{proberg}
Let $i,j\in\mathbb S^2$, let $p>0$ and  $\alpha >-1$ and $f\in\mathcal{SR}(\mathbb B)$.  Then $f\in\berg^p_{\alpha,i}$ if and only if $\berg^p_{\alpha,j}$.
\end{proposition}
\begin{proof}
Let $f\in\mathcal{SR}(\mathbb B)$ and let $w = x+yj\in\mathbb B_j$,  $z=x+yi\in\mathbb B_i$. Note that $|z|=|w|$. Then,  the Representation Formula implies
$$   | f (w)|  = \frac{1}{2}\left|(1 + ji)  f (z)      +    (1  - j i)  f (\bar z)\right| \leq \left|  f (z) \right|     +    \left|  f (\bar z)\right|. $$
This yields
\begin{gather*}  \int_{\mathbb B_j}  |  f(w)  |^p  (\alpha +1) (1-|w|^2)^{\alpha}\  d A_j (w)  \\
 \leq      2^{\max\{p-1,0\}}  (\alpha +1) \left(  \int_{\mathbb B_i} \  |   ( 1-|z|^2) ^{\alpha}   f (z)  |^p  \ d A_i (z)   +
 \int_{\mathbb B_i} \  |   ( 1-|z|^2) ^{\alpha} f( \bar z)  |^p  \ d A_i  (z)  \right).    \end{gather*}
Moreover, the change of  coordinates $\bar z \to z $ gives
$$ \int_{\mathbb B_i} \  |   ( 1-|z|^2) ^{\alpha} f(\bar z)  |^p  \ dA_i (z)  =
 \int_{\mathbb B_i} \  |   ( 1-|z|^2) ^{\alpha}   f( z)  |^p  \ d A_i (z),$$
and so
$$  \int_{\mathbb B_j}  |   ( 1-|w|^2) ^{\alpha} f(w)  |^p  (\alpha +1)\ dA_j  (w)    \leq        2^{\max\{p,1\}}\int_{\mathbb B_i} |   ( 1-|z|^2) ^{\alpha}   f(z)  |^p  (\alpha +1)\ d A_i  (z). $$
Thus, for any $f\in\berg^p_{\alpha,i}$ we have that $f\in \berg^p_{\alpha,i}$. By exchanging the roles of $i$ and $j$, we obtain the other inclusion.
\end{proof}
Next results describe the growth of slice regular functions in the ball and, in particular, of functions belonging to the weighted Bergman spaces.
\begin{proposition}
Let $p>0$, let $i\in\mathbb S^2$ and let $f\in \mathcal {SR}(\mathbb B)$.
\begin{enumerate}[(i)]
\item For any $0<r<1$ we have $$\displaystyle |f(0)|^p\leq \frac{2^{\max\{p,1\}}}{ 2\pi}   \int_0^{2\pi} |f(re^{i\theta})|^p d\theta. $$
\item For $\alpha > -1$ we have $$\displaystyle |f(0)|^p\leq 2^{\max\{p,1\}}   \int_{\mathbb{B}_i} |f(z)|^p d A_{ i, \alpha }(z). $$
\end{enumerate}
\end{proposition}
\begin{proof}
Let $j \in\mathbb{S}^2$ such that $i\perp j$ and let $f_1,f_2\colon\mathbb B_i\to\mathbb C(i)$ be holomorphic functions such that $ Q_i[f] = f_1 + f_2j$. Then, as a direct consequence of the Lemma 4.11 in \cite{kehe} and of the inequality \eqref{eq_bergmanf1f2} in the previous remark, one obtains
\begin{align*}
 |   f(0) |^p &  \leq 2^{\max\{p-1,0\}}\left(   |f_1(0)|^p  +    |f_2(0)|^p \right) \\
&\leq 2^{\max\{p-1,0\}}\left(\frac{1}{2\pi}\int_{0}^{2\pi}|f_1(re^{i\theta})|^p d\theta + \frac{1}{2\pi}\int_{0}^{2\pi}|f_2(re^{i\theta})|^p d\theta\right)\\
& \leq \frac{2^{\max\{p,1\}}}{ 2\pi}   \int_0^{2\pi} |f(re^{i\theta})|^p d\theta.
\end{align*}
The second inequality follows in the same way from the analogous result for holomorphic functions, Lemma 4.12 in \cite{kehe}.

\end{proof}

\begin{proposition} Let $p>0$ and $\alpha > - 1$ and let $f\in\berg^p_{\alpha,i}$ with $i\in\mathbb S^2$. Then for any $z\in\mathbb B_i$ we have
 \begin{equation}\label{eqprop4.6}
 |f(z) | \leq \frac{2\| f \|_{p,\alpha, i} }{(1-|z|^2)^{\frac{(2+\alpha)}{p} }}.
\end{equation}
\end{proposition}
\begin{proof}
  Let $z\in\mathbb B$ and $j \in\mathbb{S}^2$ such that $i\perp j$ and let $f_1,f_2\colon\mathbb B_i\to\mathbb C(i)$ be holomorphic functions such that $ Q_i[f] = f_1 + f_2j$. Then, by Theorem 4.14 in \cite{kehe}, we have
\begin{align*}
|f(z)|& \leq |f_1(z)| + |f_2(z)|
\\
&
\leq \frac{\| f_1\|_{p,\alpha}}{(1-|z|^2)^{\frac{2+\alpha}{p}}} + \frac{\| f_2\|_{p,\alpha}}{(1-|z|^2)^{\frac{2+\alpha}{p}}}
\\
&
\leq \frac{2\| f\|_{p,\alpha,i}}{(1-|z|^2)^{\frac{2+\alpha}{p}}}.
\end{align*}
\end{proof}
\begin{corollary}
Let $p>0$ and $\alpha > - 1$ and let $f\in\berg^p_{\alpha}$. Then
$$ |f(q) | \leq  \frac{4\, \| f \|_{p,\alpha} }{(1-|q|^2)^{\frac{(2+\alpha)}{p} }}
 $$
for all $q\in\mathbb B$.
\end{corollary}
\begin{proof} The result follows by the Representation Formula and taking the supremum for $i\in\mathbb S^2$ in (\ref{eqprop4.6}).
\end{proof}

\begin{proposition}Let $p>0$ and $\alpha > - 1$. Then
 $$\sup\ \{   |f(q)| \colon  \| f  \|_{p,\alpha} \leq 1, \, q\in S \}< \infty $$
 for every $f\in\berg^p_{\alpha}$ and for any axially symmetric  compact set $S\subset\mathbb B$.
\end{proposition}
\begin{proof}
Let $i,j\in\mathbb S^2$ be such that $i\perp j$ and let $S_i=S\cap\mathbb C(i)$. Since $\| f\|_{p,\alpha}\leq 1$ then, by definition,  $\| f\|_{p,\alpha,i}\leq 1$ and $f\in\berg^p_{\alpha,i}$. Let $f_1,f_2\colon\mathbb B_i\to\mathbb C(i)$ be holomorphic functions such that $ Q_i[f] = f_1 + f_2j$. Then, from the corresponding result in the complex case, Corollary 4.15 in \cite{kehe}, we obtain
\begin{multline*}
\sup\{   |f(z)| \colon \| f  \|_{p,\alpha, i} \leq 1, z\in S_i  \}\leq\\
 \sup\{   |f_1(z)| \colon  \| f_1  \|_{p,\alpha} \leq 1, z\in S_i\}+\sup\{   |f_2(z)| \colon\| f_2  \|_{p,\alpha} \leq 1, z\in S_i  \}  \ < \infty.
\end{multline*}
The result now follows from the Representation Formula, in fact, for any $q=x+I_qy\in S$, $z=x+iy,\bar z=x-iy\in S_i$ we have
$$
|f(q)|\leq |f(z)| +|f(\bar z)| <\infty,
$$
thus
\begin{multline*}
\sup\{   |f(q)| \colon \| f  \|_{p,\alpha} \leq 1, q\in S  \}\leq\\
 \sup\{   |f(z)| \colon  \| f  \|_{p,\alpha,i} \leq 1, z\in S_i\}+\sup\{   |f(\bar z)| \colon\| f  \|_{p,\alpha,i} \leq 1, z\in S_i  \}  \ < \infty,
\end{multline*}
and this concludes the proof.
\end{proof}

\begin{proposition}
The space $\berg^p_{\alpha}$, $p>0$ and $\alpha > - 1$, is complete.
\end{proposition}
\begin{proof}
Let $f_n$ be a Cauchy sequence in $\berg_{\alpha}^p$. Then $f_n$ is a Cauchy sequence in $\berg_{\alpha,i}^p$ for some $i\in\mathbb S^2$. Let $j\in\mathbb S^2$ be such that $i\perp j$ and let $f_{n,1},f_{n,2}\colon\mathbb B_i\to\mathbb C(i)$ be holomorphic functions such that $ Q_i[f_n] = f_{n,1} + f_{n,2}j$.

 As the complex Bergman space $\berg_{\alpha,\mathbb C}^p$ is complete, see Corollary 4.16 in \cite{kehe}, there exist functions $f_{1},f_{2}\in\berg_{\alpha,\mathbb C}^p$ such that $f_{n,1}\to f_1$ and $f_{n,2}\to f_2$ in $\berg_{\alpha,\mathbb C}^p$ as $n\to\infty$. Let $f = P_i(f_{1} + f_{2}j)$, then we have $f\in\berg^p_{\alpha,i}$. Furthermore, as
$$\|f_n - f\|_{p,\alpha,i} \leq \| f_{n,1} -f_1\|_{p,\alpha} + \|f_{n,2} - f_2\|_{p,\alpha} \overset{n\to\infty}{\longrightarrow} 0,$$
for $p\geq 1$ and
$$\|f_n - f\|_{p,\alpha,i}^p \leq \| f_{n,1} -f_1\|_{p,\alpha}^p + \|f_{n,2} - f_2\|_{p,\alpha}^p \overset{n\to\infty}{\longrightarrow} 0,$$
for $0<p\leq1$,
we have $f_n\to  f$ in $\berg_{\alpha,i}^p$. Thus, $f\in \berg_{\alpha,i}^p$ but also $f\in\berg_{\alpha}^p$  which is then complete.
\end{proof}

In complex analysis, the distance function in the Bergman metric on the unit disc $\mathbb D$ is
$$
\beta(z,w) = \frac{1}{2}\log\frac{1+\rho(z,w)}{1-\rho(z,w)},\ \ \ {\rm for\ all } \ z,w\in\mathbb D,
$$
 where $\rho(z,w) = \left| \frac{z-w}{1-z\bar{w}}\right|$. This motivates the following definition.
\begin{definition}
For $i\in\mathbb S^2$ we define
$$\beta_i(z,w) = \frac{1}{2}\log\frac{1+\rho(z,w)}{1-\rho(z,w)},\qquad z,w\in\mathbb B_i,$$
 where $$\rho(z,w) = \left| \frac{z-w}{1-z\bar{w}}\right|.$$
\end{definition}
Note that, as in the complex case, $\beta_i$ is invariant under any slice regular M\"obius transformation $T_a$ with $a \in \mathbb B_i$, that is $\beta_i(T_a(z),T_a(w)) = \beta_i(z,w)$ for all $z, w \in\mathbb B_i.$

\begin{proposition}
Let $p>0$, $\alpha > -1$ and $r>0$. For $z\in\mathbb B_i$, let $D_i(z,r) = \{w\in\mathbb B_i\colon \beta_i(z,w)<r\}$. Then there exists a positive constant $C$ such that, for all $f\in\mathcal{SR}(\mathbb B)$ and all $z\in\mathbb B_i$, we have
$$|f(z)|^p \leq \frac{2^{\max\{p,1\}}C}{(1-|z|^2)^{2+\alpha}}\int_{D_i(z,r)}|f(w)|^p dA_{\alpha,i}(w).$$
\end{proposition}
\begin{proof}
Owing to Proposition 4.13 in \cite{kehe}, there exists a constant $C$ such that $$|g(z)|^p \leq \frac{ C}{(1-|z|^2)^{2+\alpha}}\int_{D_i(z,r)}|g(w)|^p dA_{\alpha,i}(w)$$ for all $z\in\mathbb B_i$ and all holomorphic functions $g\colon\mathbb B_i\to\mathbb C(i)$. Now let $j\in\mathbb S^2$ be such that $i\perp j$ and for any $f\in\mathcal{SR}(\mathbb B)$ let $f_1,f_2\colon\mathbb B_i\to\mathbb C(i)$ be holomorphic functions such that $ Q_i[f] = f_1 + f_2j$. Then, because of the inequality \eqref{eq_bergmanf1f2} in Remark \ref{bergmanf1f2}, we have
\begin{align*}
|f(z)|^p&\leq 2^{\max\{p-1,0\}} (|f_1(z)|^p + |f_2(z)|^p)  \\
&\leq2^{\max\{p-1,0\}}\Big(\frac{C}{(1-|z|^2)^{2+\alpha}}\int_{D_i(z,r)}|f_1(w)|^p dA_{\alpha,i}(w)
\\
&+ \frac{C}{(1-|z|^2)^{2+\alpha}}\int_{D_i(z,r)}|f_2(w)|^p dA_{\alpha,i}(w)\Big)
\\
&\leq \frac{2^{\max\{p,1\}}C}{(1-|z|^2)^{2+\alpha}}\int_{D_i(z,r)}|f(w)|^p dA_{\alpha,i}(w).
\end{align*}
\end{proof}

\section{Besov spaces}\label{BSpaces}

In complex analysis, the Besov spaces are analogue to Bergman spaces and they are studied with similar techniques.
The Besov space $\besov_{p, \mathbb{C}}$, for $p>0$, is defined as the set of all analytic functions on the complex unit disc $\mathbb D$ such that
$$
\int_{\mathbb D}|(1-|z|^2)^{n}f^{(n)}(z)|^pd\lambda(z)<\infty,
$$
 for some $n$ with $np>1$ where $d\lambda(z) = \frac{dA(z)}{(1-|z|^2)^2}$.

To move to the setting of quaternionic Besov spaces, we define a suitable differential.
For $i\in \mathbb S^2$ let $$ d\lambda_i(z) = \frac{dA_i(z)}{ (1-|z|^2) ^2},$$
where  $dA_i(z)$ is again the normalized differential of area in the plane $\mathbb C(i)$. Note that, as in the complex case, $d\lambda_i$ is invariant under slice regular M\"obius transformations $T_a$ with $a\in\mathbb B_i$.

\begin{proposition}\label{pro1}
Let $i\in\mathbb S^2$, let $p>0$ and let $m,n \in \mathbb N$ such that $np>1$ and $mp>1$.  Then for any $f\in \mathcal{SR}(\mathbb B)$ the inequality
$$ \int_{\mathbb B_i} \  |   ( 1-|z|^2) ^n  \partial_{x_0}^{n}  f   (z)  |^p  \ d \lambda_i(z)    \  < \infty  $$
holds if and only if
$$ \int_{\mathbb B_i} \  |   ( 1-|z|^2)^m  \partial_{x_0}^{m} f (z)  |^p  \ d \lambda_i(z)    \  < \infty.$$
\end{proposition}
\begin{proof}
Let $f\in \mathcal{SR}(\mathbb B)$, let $j\in\mathbb S^2$ such that $i\perp j$ and let $f_1,f_2\colon\mathbb B_i\to\mathbb C(i)$ be holomorphic functions such that $ Q_i[f] = f_1 + f_2j$. Note that ${Q}_i[f]$ is a function of the complex variable $z = x_0 + yi$. So for all $k\in\mathbb N$ we have $\partial_{x_0}^k f(z) = Q_i[f]^{(k)}(z) = f_1^{(k)}(z) + f_2^{(k)}(z)j$ and, from the inequality \eqref{eq_bergmanf1f2} in Remark \ref{bergmanf1f2}, we obtain
\begin{align*} |(1-|z|^2)^kf_{l}^{(k)}(z)|^p &\leq |(1-|z|^2)^k \partial_{x_0}^kf (z)|^p\\
&\leq 2^{\max\{p-1,0\}}|(1-|z|^2)^kf_{1}^{(k)}(z)|^p + 2^{\max\{p-1,0\}} |(1-|z|^2)^kf_{2}^{(k)}(z)|^p\end{align*}
for $l=1,2$ and $z\in\mathbb B_i$. Therefore, the inequality
\begin{equation}\label{int_d0nC} \int_{\mathbb B_i} \  |   ( 1-|z|^2) ^k   f_l^{(k)}   (z)  |^p  \ d \lambda_i(z) < \infty,  \qquad l=1,2,
\end{equation}
holds if and only if $ \int_{\mathbb B_i} \  |   ( 1-|z|^2) ^k  \partial_{x_0}^{k}  f   (z)  |^p  \ d \lambda_i(z)    \  < \infty$. But from Lemma 5.16 in \cite{kehe}, we have that \eqref{int_d0nC} holds for $k=n$ if and only if it holds for $k=m$, thus the statement is true.

\end{proof}

 The following definition will be useful in the sequel.
\begin{definition}
By $\besov_{p,i}$ we denote the quaternionic right linear space of slice regular functions on $\mathbb B$ such that
$$ \int_{\mathbb B_i} \  |   ( 1-|z|^2) ^n  \partial_{x_0}^{n}   f(z)  |^p  \ d \lambda_i (z)    \  < \infty ,
$$
for some $n\in\mathbb N$ such that $np>1$.
Moreover we set  $\besov_{\infty,i}=\mathcal{B}_{i}$, where $\mathcal{B}_{i}$ is as in Definition \ref{BICONI}.
\end{definition}
Note that by Proposition \ref{pro1} this definition is independent of the choice of $n$.
\begin{remark}\label{besovf1f2}
{\rm
Let $j\in\mathbb S^2$ be such that $j\perp i$. Then there exist holomorphic functions $f_1,f_2\colon\mathbb B_i\to\mathbb C(i)$ such that $ Q_i[f] = f_1 + f_2j$ and so $\partial_{x_0}^nf(z) = f_1^{(n)}(z) + f_2^{(n)}(z)j$ for $z\in\mathbb B_i$. Thus, for $z\in\mathbb B_i$, we have
$$| f_k^{(n)}(z)|^p\leq | \partial_{x_0}^nf(z)|^p \leq 2^{\max\{0,p-1\}}\left(|f_1^{(n)}(z)|^p + | f_2^{(n)}(z)|^p\right),\qquad k=1,2.$$
Thus, the function $f$ is in $\besov_{p,i}$ if and only if $f_1$ and $f_2$ are in the complex Besov space $\besov_{p, \mathbb{C}}$.
}
\end{remark}

\begin{proposition}\label{prp413}
Let $p>0$ and  $i,j\in \mathbb S^2$. Then  $\besov_{p,i}$ and $\besov_{p,j}$ are equal as sets.
\end{proposition}
\begin{proof}
Let $f\in\mathcal{SR}(\mathbb B)$ and let $n\in\mathbb N$ such that $pn >1$. For $w = x_0+yj\in\mathbb B_j$, let $z=x_0+yi\in\mathbb B_i$. Then, from the Representation Formula, one obtains
$$   |\partial_{x_0}^{n} f (w)|  = \frac{1}{2}\left|(1 - ji) \partial_{x_0}^{n} f (z)      +    (1  + j i)   \partial_{x_0}^{n} f (\bar z)\right| \leq \left|  \partial_{x_0}^{n} f (z) \right|     +    \left| \partial_{x_0}^{n} f (\bar z)\right|. $$
This implies
\begin{gather*}  \int_{\mathbb B_j}   |   ( 1-|w|^2) ^n  \partial_{x_0}^{n}  f(w)  |^p\  d \lambda_j (w)  \\
 \leq      2^{\max\{p-1,0\}}  \left(  \int_{\mathbb B_i} \  |   ( 1-|z|^2) ^n  \partial_{x_0}^{n} f (z)  |^p  \ d \lambda_i (z)   +
 \int_{\mathbb B_i} \  |   ( 1-|z|^2) ^n  \partial_{x_0}^{n} f( \bar z)  |^p  \ d\lambda_i  (z)  \right).    \end{gather*}
By changing  coordinates $\bar z \to z $, we obtain
$$ \int_{\mathbb B_i} \  |   ( 1-|z|^2) ^n  \partial_{x_0}^{n} f(\bar z)  |^p  \ d\lambda_i (z)  =
 \int_{\mathbb B_i} \  |   ( 1-|z|^2) ^n  \partial_{x_0}^{n}  f( z)  |^p  \ d \lambda_i (z),$$
and so
$$  \int_{\mathbb B_j}  |   ( 1-|w|^2) ^n  \partial_{x_0}^{n} f(w)  |^p  \ d\lambda_j  (w)    \leq        2^{\max\{p,1\}}\int_{\mathbb B_i} |   ( 1-|z|^2) ^n  \partial_{x_0}^{n}  f(z)  |^p  \ d \lambda_i  (z). $$
Thus, for any $f\in\besov_{p,i}$ we have that $f\in\besov_{p,j}$. By exchanging the roles of $i$ and $j$, we obtain the other inclusion.
\end{proof}
By virtue of the previous result we can now give the definition of Besov space in this setting:
\begin{definition}
Let $p>0$, let $i\in \mathbb S^2$ and let $n\in \mathbb N$ with $pn>1$. The {\em slice regular Besov space $\besov_{p}$} is the quaternionic right linear space of slice regular functions on $\mathbb B$ such that
$$ \sup_{i\in\mathbb S^2}\int_{\mathbb B_i} \  |   ( 1-|z|^2) ^n  \partial_{x_0}^{n}   f(z)  |^p  \ d \lambda_i (z)    \  < \infty .
$$
We also define the Besov space  $\besov_{\infty}=\mathcal{B}$, where $\mathcal{B}$ is the Bloch space.
\end{definition}
Next result gives a nice characterization of the functions in $ \besov_{p}$.
\begin{proposition}\label{pro2}
Let $i\in\mathbb S^2$. For any $p>0$ there exists a sequence $(a_k)_{k\geq 1}$ in $\mathbb B_i$ with the following property:  if $ b>\max\{0, (p-1)/p\} $, then $f\in \besov_{p}$ if and only if there exists a sequence $(d_k)_{k\geq 1}\in \ell^p(\mathbb H)$ such that
\begin{equation}\label{besovrepH}f = \sum_{k=1}^{\infty} P_i\left[\left(\frac{1-|a_k|^2}{1-z\overline{a_k}}\right)^b\right]d_k.\end{equation}
\end{proposition}
\begin{proof}
Let $i\in\mathbb S^2$ and let us identify $\mathbb B_i$ with the unit disc $\mathbb D$.
Theorem 5.17 in \cite {kehe} yields the existence of a sequence $(a_k)_{k\geq1}\subset\mathbb B_i $ such that the complex Besov space
$\besov_{p, \mathbb{C}}$ on $\mathbb B_i$ consists exactly of the functions of the form
\begin{equation}\label{besovrepC}  g (z) = \sum_{k=1}^{\infty}   \left(\frac{1-|a_k|^2}{1-z\overline{a_k} } \right)^b c_k\end{equation}
with $(c_k)_{k\geq1}\in\ell^p(\mathbb C(i))$.
Now let $j\in \mathbb S^2$ such that $j\perp i$. For any $f\in\besov_p$
 there exist two holomorphic functions $f_1, f_2\colon\mathbb B_i\to\mathbb C(i)$ such that $ Q_i[f]= f_1+ f_2 j$.
Since $f\in \besov_{p,i}$, for $i\in\mathbb S^2$, from Remark \ref{besovf1f2} it follows that $f_1$ and $f_2$ belong to the complex Besov space $\besov_{p, \mathbb{C}}$ and so there exist sequences $(c_{1,k})_{k\geq1},(c_{2,k})_{k\geq1}\in\ell^p(\mathbb C(i))$ that give the representation \eqref{besovrepC} of $f_1$ and $f_2$, respectively. Therefore, we have
$$ Q_i[f](z)= f_{1} (z) +f_2(z) j = \sum_{k=1}^{\infty}  \left(\frac{1-|{a_k}|^2}{1-z\overline{a_k} } \right)^b c_{1,k}   +
   \sum_{k=1}^{\infty}\left(\frac{1-|a_k|^2}{1-z\overline{a_k} } \right)^b c_{2,k}  j. $$
As $P_i\circ Q_i= I_{\mathcal{SR}(\mathbb B)}$, we obtain the desired representation
$$ f = P_i\circ Q_i [f] = P_i \left[\sum_{k=1}^{\infty} \left(\frac{1-|a_k|^2}{1-z\overline{a_k} } \right)^b(c_{1,k}  +c_{2,k}  j)\right] = \sum_{k=1}^{\infty}P_i \left[ \left(\frac{1-|a_k|^2}{1-z\overline{a_k} } \right)^b\right](c_{1,k}  +c_{2,k}  j), $$
where the sequence of the coefficients $d_k = c_{1,k} + c_{2,k}j$ lie in $\ell^p(\mathbb H)$ because $(c_{1,k})_{k\geq 1}$ and $(c_{2,k})_{k\geq1}$ belong to $\ell^p(\mathbb C(i))$.

If, on the other hand, $f$ has the form \eqref{besovrepH}, then there exist sequences $(c_{1,k})_{k\geq1}$ and $(c_{2,k})_{k\geq1}$ such that $d_k = c_{1,k} + c_{2,k}j$.  Thus, we have
$${Q}_i[f] = f_1 + f_2 j = \sum_{k=1}^{\infty}   \left(\frac{1-|{a_k}|^2}{1-z\overline{a_k} } \right)^b c_{1,k}+ \sum_{k=1}^{\infty}   \left(\frac{1-|{a_k}|^2}{1-z\overline{a_k} } \right)^b c_{2,k} \,j
$$
and as $(d_k)_{k\geq 1}$ lies in $\ell^p(\mathbb H)$, it follows that $(c_{1,k})_{k\geq1}$ and $(c_{2,k})_{k\geq 1}$ are in $\ell^p(\mathbb C(i))$. Therefore, $f_1$ and $f_2$ are of the form \eqref{besovrepC} and so they lie in the complex Besov space $\besov_{p, \mathbb{C}}$, which implies $f\in\besov_{p,i}$ and since $i\in\mathbb S^2$ is arbitrary, we conclude that $f\in\besov_{p}$.

\end{proof}

The following result shows that  the space $\besov_{p,i}$ is invariant under M\"obius transformation if one takes the $\circ_i$ composition.

\begin{proposition}Let $i\in \mathbb S^2$, let $a\in\mathbb B_i$ and let $T_a$ be a slice regular M\"obius transformation. Then for $f\in\besov_{p,i}$ we have $f\circ_iT_a\in\besov_{p,i}$.
\end{proposition}
\begin{proof}
Let $j\in\mathbb S^2$ with $j\perp i$ and let $f_1,f_2\colon\mathbb B_i\to\mathbb C(i)$ be holomorphic functions such that $ Q_i[f] = f_1 + f_2j$. Moreover, the functions $f_1$ and $f_2$ lie in the complex Besov space $\besov_{p, \mathbb{C}}$ because of Remark \ref{besovf1f2}. By the definition of the $\circ_i$-composition, we have
$$ {Q}_i[f\circ_i T_a] = f_1\circ Q_i[T_a] + f_2\circ Q_i[T_a]j.$$
But the function $ Q_i[T_a]$ is nothing but the complex M\"obius transformation associated with $a\in \mathbb B_i$. Moreover, because of Theorem 5.18 in \cite{kehe}, the complex Besov space $\besov_{p, \mathbb{C}}$ is M\"obius invariant. Thus, $f_1\circ Q_i[T_a]$ and $f_2\circ Q_i[T_a]$ lie in the complex Besov space, which is equivalent to $f\circ_iT_a\in\besov_{p,i}$.

\end{proof}

\begin{definition}\label{Defrho}
Let $i\in \mathbb S^2$. For  $p>1$, we define
$$\rho_{p,i}(f) =\left[    \int_{\mathbb B_i}     (1-|z|^2)^p |\partial_{x_0} f(z)|^p      d\lambda_i (z)      \right]^{\frac{1}{p}}  $$
for all $f\in \besov_{p,i}$. For $0< p \leq 1$ and $n\in\mathbb N$ with $np>1$, we define
$$\rho_{p,i,n}(f) = \sup_{z\in\mathbb B_i } |f(z)|   +  \sup_{a\in\mathbb B_i } \left[ \int_{\mathbb  B_i}     (1-|z|^2)^{np} |\partial_{x_0}^{n} (f\circ_i T_a)(z)|^p   d \lambda_i(z)   \right]^{\frac{1}{p}}   , $$
for all $f\in\besov_{p,i}$.\end{definition}

In the following result we show that, for $p>1$, $\besov_{p,i}$ is a Banach space invariant under the $\circ_i$-composition with a M\"obius map.

\begin{lemma}\label{lemma4.17}
Let $i\in\mathbb S^2$ and $p>1$. Then $\rho_{p,i}$ is a complete seminorm on $\besov_{p,i}$ such that  $\besov_{p,i}$ modulo the constant functions endowed with the norm induced by $\rho_{p,i}$ is a Banach space, and it satisfies
$$\rho_{p,i}(f\circ_iT_a) = \rho_{p,i}(f) $$
for all $f\in\besov_{p,i}$ and all $a\in\mathbb B_i$.
\end{lemma}
\begin{proof}
It is clear, that $\rho_{p,i}$ actually is a seminorm. Let $j\in\mathbb S^2$ with $i\perp j$ and for $f\in\besov_{p,i}$ let $f_1,f_2\in\besov_p$ such that $ Q_i[f] = f_1 + f_2j$. Then, the inequality
$$|f'_k(z)|^p \leq |\partial_{x_0}f(z)|^p \leq 2^{p-1}\left(|f'_1(z)|^p + |f'_2(z)|^p\right)\quad \text{for }k=1,2, $$
implies
$$\rho_{p}(f_k)^p \leq \rho_{p,i}(f)^p \leq 2^{p-1}\big(\rho_{p}(f_1)^p + \rho_{p}(f_2)^p\big)\quad \text{for }k=1,2, $$
where $\rho_{p}$ is the corresponding seminorm on the complex Besov space $\besov_{p, \mathbb{C}}$ on $\mathbb B_i$. Thus, as $\rho_{p}$ is a complete seminorm on the complex Besov space (see Theorem 5.18 in \cite{kehe}), the seminorm $\rho_{p,i}$ on is complete too.

Now we want to prove the M\"obius invariance of $\rho_{p,i}$, so let $a\in\mathbb B_i$. If we denote
$$(f\circ T_a)' = \partial_{0}  Q_i[f\circ_i T_a], \qquad f' = \partial_{x_0}  Q_i[f],\qquad\text{and}\qquad T_a' = \partial_{x_0} Q_i[T_a] , $$
we have
\[
\begin{split}
 |( f\circ T_a)'(z)| &=| (f_1\circ T_a)'(z ) +  (f_2\circ T_a)'(z ) j|
 \\
 &
 =|T_a'(z)|\cdot | f_1'( T_a(z ) )+  f_2' ( T_a(z ))j |
 \\
 &
 = |T_a'(z)|\cdot |f'(T_a(z))|.
 \end{split}
 \]
Now recall that $(1-|z|^2)|T_a'(z)| = 1- |T_a(z)|^2$ for $z\in\mathbb B_i$ as $ Q_i[T_a]$ is the usual complex M\"obius transformation on $\mathbb B_i$ associated with $a$. So, as $\lambda_i$ is invariant under $T_a$, we obtain
\begin{align*}
\rho_{p,i}(f\circ_iT_a)^p &= \int_{\mathbb B_i}     (1-|z|^2)^p |(f\circ T_a)'(z)|^p   \,   d\lambda_i (z)\\
& = \int_{\mathbb B_i}     (1-|z|^2)^p |T_a'(z)|^p\cdot |f'(T_a(z))|^p   \,   d\lambda_i (z)\\
&= \int_{\mathbb B_i}     (1-|T_a(z)|^2)^p |f'( T_a(z))|^p   \,   d\lambda_i (T_a(z)) = \rho_{p,i}(f)^p.
\end{align*}
\end{proof}

\begin{corollary}
Let $p>1$ and let $\rho_p(f)=\sup_{i\in\mathbb S^2}\rho_{p,i}(f)$. Then $\rho_p$ is a complete seminorm on $\besov_{p}$ such that  $\besov_{p}$ modulo the constant functions endowed with the norm induced by $\rho_{p}$ is a Banach space.
\end{corollary}
\begin{proof}
The result follows from the previous Lemma.
\end{proof}
\begin{remark}{\rm  We observe that, in general, $\rho_p(f\circ_iT_a)^p\neq \rho_{p}(f)^p$. In fact
for $w= x_0 + j x_1$ let $z_i=x_0 + i x_1$ and assume that $f'$ denotes
the derivative the restriction of $f$ to the plane $\mathbb C(i)$, then
\begin{align*}
\rho_p(f\circ_iT_a)^p &= \sup_{j\in\mathbb S} \rho_{p,j}(f\circ_iT_a) ^p = \sup_{j\in\mathbb S} \int_{\mathbb B_j}(1-|w|^2)^p|\partial_0f\circ_iT_a|^p d\lambda_j(w)\\
&=\sup_{j\in\mathbb S}\int_{\mathbb B_i}\frac{1}{2^p}(1-|z |^2)^p\big|(1-ji)(f\circ_iT_a)'(z ) + (1+ji)(f\circ_iT_a)'(\bar{z})\big|^p\,d\lambda_i(z) \\
&=\sup_{j\in\mathbb S}\int_{\mathbb B_i}\frac{1}{2^p}(1-|z|^2)^p\big|(1-ji)f'(T_a(z))T_a'(z) + (1+ji)f'(T_a(\bar{z}))T_a'(z)\big|^p\,d\lambda_i(z).
\end{align*}
But, in general, it is $|T_a'(z)| \neq |T_a'(\bar{z})|$ thus we obtain
\begin{align*} \rho_p(f\circ_iT_a)^p&\neq\sup_{j\in\mathbb S}\int_{\mathbb B_i}\frac{1}{2^p}(1-|z|^2)^p\,|T_a'(z)|^p\,\big|(1-ji)f'(T_a(z)) + (1+ji)f'(T_a(\bar{z}))\big|^p\,d\lambda_i(z) \\
&= \sup_{j\in\mathbb S}\int_{\mathbb B_i}\frac{1}{2^p}(1-|T_a(z)|^2)^p\big|(1-ji)f'(T_a(z)) + (1+ji)f'(T_a(\bar{z}))\big|^p\,d\lambda_i(T_a(z))\\
&= \sup_{j\in\mathbb S}\rho_{p,i}(f)^p = \rho_{p}(f)^p,\end{align*}
where we used the fact that $(1-|T_a(z)|^2) = (1-|z|^2)|T_a'(z)|$
and that $d\lambda_i$ is invariant under $T_a$.
}
\end{remark}

\begin{corollary}
For $p>1$, the Besov space $\besov_{p}$ is a Banach space with the norm
$$\|f\|_{\besov_{p}} = |f(0)| + \sup_{i\in \mathbb{S}^2}\left[    \int_{\mathbb B_i}     (1-|z|^2)^p |\partial_{x_0} f(z)|^p      d\lambda_i (z)      \right]^{\frac{1}{p}}.$$
\end{corollary}
\begin{proof}
It is an immediate consequence of  the above lemma and of the fact that $f\in\besov_p$ if and only if $f\in\besov_{p,i}$.
\end{proof}
\begin{definition} A seminorm $\rho$ is said to be
 $i$-M\"obius invariant on $\besov_{p,i}$ if it satisfies the condition in Lemma \ref{lemma4.17}.
\end{definition}

According to Definition \ref{Defrho}, we now study the case $0<p\leq 1$. We have the following result:
\begin{proposition}
Let $i\in\mathbb S^2$ and $0<p\leq1$ and let $n\in\mathbb N$ with $np>1$. Then   $\rho_{p,n,i}(f)$ is finite
 for any $f\in\besov_{p,i}$. Furthermore, we have
\begin{equation}\label{SNinvariant}\rho_{p,n,i}(f \circ_i T_a )=\rho_{p,n,i}(f), \quad \text{for all } a\in\mathbb B_i. \end{equation}
\end{proposition}
\begin{proof}
As for $a\in\mathbb B_i$ the slice regular M\"obius transformation maps $\mathbb B_i$ bijectively onto itself, by the definition of
$\rho_{{p,n,i}}$, it is obvious that \eqref{SNinvariant} holds, provided that $\rho_{{p,n,i}}$ is finite.

Now let $f\in\besov_{p,i}$. To show that $\rho_{p,n,i}(f)$ is finite, we chose again $j\in\mathbb S^2$ with $i\perp j$ and apply the Splitting Lemma to obtain functions $f_1$ and $f_2$ in the complex Besov space $\besov_{p, \mathbb{C}}$ on $\mathbb B_i$ such that  $ Q_i[f] = f_1 + f_2 j$.
From Theorem 5.18 in \cite{kehe}, we know that, for any function $g$ in $\besov_p$, we have
$$\rho_{p,n}(g) = \sup_{z\in\mathbb B_i} |g(z)|    +  \sup_{a\in\mathbb B_i} \left[ \int_{\mathbb  B_i}     (1-|z|^2)^{np} |( g\circ \varphi_a)^{(n)}(z)|^p   d \lambda_i(z)   \right]^{\frac{1}{p}} <\infty, $$
where $\varphi_a$ denotes the complex M\"obius transformation associated with $a$. Thus, on one hand we have
$$ \sup_{z\in\mathbb B_i}|f(z)|  \leq \sup_{z\in\mathbb B_i}|f_1(z)| +\sup_{z\in\mathbb B_i}|f_2(z)|\leq \rho_{p,n}(f_1) +\rho_{p,n}(f_2)< \infty.$$
 On the other hand, as $0<p\leq1$, we have $(\alpha+\beta)^p <\alpha^p + \beta^p$   for $\alpha,\beta>0$. Moreover, for any $a$ in $\mathbb B_i$, we have $\varphi_a =  Q_i[T_a]$. As $x\mapsto x^{1/p}$ is increasing on $\mathbb R^+$, we can exchange it with the supremum and so we obtain
 \begin{multline*}
 \sup_{a\in\mathbb B_i} \int_{\mathbb  B_i}     (1-|z|^2)^{np} |\partial_{x_0}^n( f\circ_iT_a)(z)|^p   d \lambda_i(z)   \leq \sup_{a\in\mathbb B_i} \int_{\mathbb  B_i}     (1-|z|^2)^{np} |( f_1\circ \varphi_a)^{(n)}(z)|^p   d \lambda_i(z)  \\
+  \sup_{a\in\mathbb B_i} \int_{\mathbb  B_i}     (1-|z|^2)^{np} |( f_2\circ \varphi_a)^{(n)}(z)|^p   d \lambda_i(z)  \leq \rho_{p,n}(f_1)^p +\rho_{p,n}(f_2)^p<\infty.
 \end{multline*}
Putting all this together, we obtain that $\rho_{p,n,i}(f)$ is finite.
\end{proof}

The space $\besov_{1}$ requires some special attention.

\begin{proposition}
Let $f\in\mathcal{SR}(\mathbb B)$ and let $i\in\mathbb S^2$. Then the following facts are equivalent
\begin{enumerate}[(i)]
\item \label{prop_B1_1} $f$ belongs to $\besov_{1}$.
\item \label{prop_B1_2} There exists a sequence $(\gamma_k)_{k\geq 0}\in \ell^1(\mathbb H)$ and a  sequence $(a_k)_{k\geq 1}\in\mathbb B_i$ such that
\begin{equation}\label{besov_l1sum}    f(q) = \gamma_0 + \sum_{k=1}^{\infty} T_{a_k}(q)\gamma_k.\end{equation}
 \item \label{prop_B1_3}There exists a finite $\mathbb H$-valued Borel measure $\mu$ on $\mathbb B_i$ such that
$$       f(q) =   \int_{\mathbb B_i} T_z(q)\,d\mu(z).$$
\end{enumerate}
\end{proposition}
\begin{proof}
We first show that (\ref{prop_B1_1}) implies (\ref{prop_B1_2}), so let $f\in\besov_{1}$. For $a\in\mathbb B_i$,  let $T_a(z) = \frac{a-z}{1-z\bar{a}}$ be the complex M\"obius transformation on $\mathbb B_i$ associated with $a$. Then we have
\begin{equation*}
\frac{1-|a|^2}{1-z\bar a } = 1 -  T_a(z)\bar a.
\end{equation*}
Thus, we can apply Proposition \ref{pro2} with $b=1$ and obtain sequences $(d_k)_{k\geq 1}\in \ell^1(\mathbb H)$ and $(a_k)_{k\geq 1}\in\mathbb B_i$ such that
\begin{equation*}
f = \sum_{k = 1}^{\infty} P_i\left[\frac{1-|a_k|^2}{1-z\overline{a_k}}\right]d_k = \sum_{k = 1}^{\infty} P_i\left[1 - T_{a_k}\overline{a_k}\right]d_k = \sum_{k = 1}^{\infty} d_k -  \sum_{k = 1}^{\infty} P_i[T_{a_k}]\overline{a_k}d_k.
\end{equation*}
So, if we define $\gamma_0 = \sum_{k = 1}^{\infty} d_k$ and $\gamma_k =- \overline{a_k}d_k$ for $k\geq1$, as $P_i[T_{a_k}] = T_{a_k}$, we obtain \eqref{besov_l1sum}. Moreover, we have $(\gamma_k)_{k\geq0}\in\ell^1(\mathbb H)$ because $|a_k|<1$ and so
$$ \sum_{k=0}^{\infty} |\gamma_k| = |\gamma_0| + \sum_{k=1}^{\infty} |a_k||d_k| \leq 2 \sum_{k=1}^{\infty}|d_k|<\infty.$$

 Now we will show that (\ref{prop_B1_2}) implies (\ref{prop_B1_3}), so let us suppose that (\ref{prop_B1_2}) holds.  Then there exist sequences $(\gamma_{1,k})_{k\in\mathbb N_0},(\gamma_{2,k})_{k\in\mathbb N_0} \subset\mathbb{C}(i)$ such that $\gamma_k = \gamma_{1,k} + \gamma_{2,k}j$ and, as $|\gamma_{1,k}|\leq|\gamma_{k}|$ and $|\gamma_{2,k}|\leq|\gamma_{k}|$ for $k\geq0$,  they are in $\ell^1(\mathbb C(i))$. Moreover,  let $j\in\mathbb S^2$ with $i\perp j$ and let $f_1,f_2\colon \mathbb B_i\to \mathbb C(i)$ be holomorphic functions such that $ Q_i[f] = f_1 + f_2j$. Then, as $\mathcal{Q}_i[T_a] = T_a$ for $a\in\mathbb B_i$, we have,
 $$ f_1(z) + f_2(z)j =  Q_i[f](z) = \gamma_{1,0} + \sum_{k=1}^{\infty} T_{a_k}(z)\gamma_{1,k} +\left( \gamma_{2,0} + \sum_{k=1}^{\infty} T_{a_k}(z)\gamma_{2,k}\right)j.$$
 So $f_l(z) = \gamma_{l,0} + \sum_{k=1}^{\infty} T_{a_k}(z)\gamma_{l,k}$ for $l=1,2$. By Theorem 5.19 in \cite{kehe}, there exist two finite Borel measures $\mu_1$ and $\mu_2$ on $\mathbb B_i$ with values in $\mathbb C(i)$ such that
$$ f_1(z) = \int_{\mathbb B_i}T_w(z) d\mu_1(w)\qquad\text{and}\qquad f_2(z) = \int_{\mathbb B_i}T_w(z) d\mu_2(w).$$
 Thus, if we set $\mu = \mu_1 + \mu_2j$, we have
 \begin{equation*}
 f(q) = P_i[f_1 +f_2j](q) = \int_{\mathbb B_i}T_w(q) d\mu_1(w)+  \int_{\mathbb B_i}T_w(q) d\mu_2(w)j = \int_{\mathbb B_i}T_w(q) d\mu(w).
 \end{equation*}

 Finally, we will show that (\ref{prop_B1_3}) implies (\ref{prop_B1_1}), so let us assume that (\ref{prop_B1_3}) holds. Then there exist two finite $\mathbb C(i)$-valued Borel measures $\mu_1,\mu_2$ such that $\mu = \mu_1 + \mu_2j$. For all $z\in\mathbb B_i$, we have
\begin{equation*}
f_1(z) + f_2(z)j =  Q_i[f](z) = \int_{\mathbb B_i}T_w(q) d\mu_1(w)+  \int_{\mathbb B_i}T_w(q) d\mu_2(w)j,
\end{equation*}
so $f_{l}(z) = \int_{\mathbb B_i}T_w(z) d\mu_l(w), l = 1,2$. Because of Theorem 5.19 in \cite{kehe} this implies that $f_1$ and $f_2$
lie in the complex Besov space $\besov_{1, \mathbb{C}}$ which is equivalent to $f\in\besov_{1,i}$ and so $f\in\besov_{1}$.

\end{proof}

As in the complex case, the representation \eqref{besov_l1sum} allows to define an $i$-M\"obius invariant Banach space structure on $\besov_{1,i}$.

\begin{remark}
{\rm
In the next result we introduce a norm on  $\besov_{1,i}$ which is different from the one in Definition \ref{Defrho}.
}
\end{remark}

\begin{proposition}
Let $i\in\mathbb S^2$. Then $\besov_{1,i}$ is a Banach space with the norm
\begin{equation}\label{norm_Besov1}\|f\|_{\besov_{1,i}} = \inf\left\{\sum_{k=0}^{\infty}|\gamma_k|\  \Big|\ \exists (a_k)_{k\geq1}\subset\mathbb B_i\colon  f = \gamma_0 + \sum_{k=1}^{\infty} T_{a_k}\gamma_k\right\}.\end{equation}
 Furthermore, for all $a\in\mathbb B_i$, the following equation holds
\begin{equation}\label{l1_MoebInv} \|f\circ_iT_a\|_{\besov_{1,i}} = \|f\|_{\besov_{1,i}}.\end{equation}
\end{proposition}
\begin{proof}
It is clear that $\| f \lambda \|_{\besov_{1,i}} = \| f \|_{\besov_{1,i}}|\lambda|$ for any $\lambda\in\mathbb H$ and that $\| f \|_{\besov_{1,i}} \geq 0$ with equality if and only if $f=0$.  Now let $f,g\in\besov_{1,i}$ and let $ f +g = \gamma_0 + \sum_{k=1}^{\infty} T_{a_k}\gamma_k$. Then $ f = \omega_0 + \sum_{k=1}^{\infty} T_{a_k}\omega_k$  if and only if $ g = \upsilon_0 + \sum_{k=1}^{\infty} T_{a_k}\upsilon_k$ with $\upsilon_k = \gamma_k - \omega_k$ and so we have
\begin{align}
\notag \|f+&g\|_{\besov_{1,i}}= \inf\left\{\sum_{k=0}^{\infty} |\omega_k + \upsilon_k|\  \Big|\ \exists (a_k)_{k\geq1}\subset\mathbb B_i\colon f +g = \omega_0 + \upsilon_0+ \sum_{k=1}^{\infty}T_{a_k}(\omega_k + \upsilon_k)\right\}
\\
&
\label{B1triangle1}\leq \inf\left\{\sum_{k=0}^{\infty} |\omega_k| + \sum_{k=0}^{\infty} |\upsilon_k|\  \Big|\ \exists (a_k)_{k\geq1}\subset\mathbb B_i\colon f = \omega_0+ \sum_{k=1}^{\infty}T_{a_k}\omega_k,\ g = \upsilon_0 + \sum_{k=1}^{\infty}T_{a_k}\upsilon_k\right\}.
\end{align}
On the other hand, if $f = \omega_0 + \sum_{k=1}^{\infty} T_{b_k}\omega_k$ and $g = \upsilon_0  + \sum_{k=1}^n T_{c_k}\upsilon_k$, then we can set $\gamma_0 = \omega_0 + \upsilon_0$ and
\begin{equation*}\left\{\begin{array}{lll}\gamma_{k} = \omega_n & a_k = b_n & \text{if }k = 2n\\
\gamma_k = \upsilon_n & a_k = c_n & \text{if }k = 2n-1
 \end{array}\right. \quad\text{for } n=1,2,3,\ldots ,\end{equation*}
and obtain $f+g = \gamma_0 + \sum_{k=1}^{\infty}T_{a_k}\gamma_k$. Thus, in \eqref{B1triangle1}, we can vary $\omega_k$ and $\upsilon_k$ independently, which allows us to split the infimum. So we obtain
\begin{multline*}
 \|f+g\|_{\besov_1,i}\leq \inf\left\{\sum_{k=0}^{\infty} |\omega_k|\  \Big|\ \exists (b_k)_{k\geq1}\subset\mathbb B_i\colon f = \omega_0+ \sum_{k=1}^{\infty}T_{b_k}\omega_k\right\}\\
  +\inf\left\{ \sum_{k=0}^{\infty} |\upsilon_k|\  \Big|\ \exists (c_k)_{k\geq1}\subset\mathbb B_i\colon  g = \upsilon_0 + \sum_{k=1}^{\infty}T_{c_k}\upsilon_k\right\} =  \|f\|_{\besov_1,i} + \|g\|_{\besov_1,i}.
\end{multline*}
Hence $\|\cdot\|_{\besov_{1,i}}$ is actually a norm.

Furthermore, for $a,b\in\mathbb B_i$ the $i$-composition $T_b\circ_iT_a$ is again a slice regular M\"obius transformation $T_c$ with $c\in\mathbb B_i$. Thus, $f = \gamma_0 + \sum_{k=1}^{\infty} T_{b_k}\gamma_k$ if and only if $f\circ_i T_a = \gamma_0 + \sum_{k=1}^{\infty}T_{c_k}\gamma_k$, where $T_{c_k} = T_{b_k}\circ T_{a}$. From the definition of $\|\cdot\|_{\besov_{1,i}}$, it is therefore clear that this norm is invariant under M\"obius transformations $T_a$ with $a\in\mathbb B_i$.

Finally, to show that the space is complete, we choose again $j\in\mathbb S^2$ such that $i\perp j$. Then, for any $\gamma_k\in\mathbb H$ there exist $\gamma_{k,1},\gamma_{k,2}\in\mathbb C(i)$ such that $\gamma_k = \alpha_k + \beta_kj$ and for any $f\in\besov_{1,i}$ there exist $f_1,f_2$ in the space $\besov_{1,i}$ on $\mathbb B_i$ such that $f = f_1 + f_2j$.
Furthermore, if $f = \lambda_{0} + \sum_{k=1}^{\infty}T_{a_k}\lambda_{k}$, then we have $f_{1} = \alpha_{0} + \sum_{k=1}^{\infty}\varphi_{a_k}\alpha_{k}$  and $f_{2} = \beta_{0} + \sum_{k=1}^{\infty}\varphi_{a_k}\beta_{k}$, where $\varphi_{a_k}$ denotes the complex M\"obius transformation on $\mathbb B_i$ associated with $a_k$.
Therefore, we get
\begin{align}&\|f\|_{\besov_1,i} \notag= \inf\left\{\sum_{k=0}^{\infty}|\alpha_k + \beta_kj|\  \Big|\ \exists (a_k)_{k\geq1}\subset\mathbb B_i \colon f = \alpha_k + \beta_0 j + \sum_{k=1}^{\infty} T_{a_k}(\alpha_k + \beta_k j)\right\}\\
&\leq \inf\left\{\sum_{k=0}^{\infty}|\alpha_k| +\sum_{k=0}^{\infty} |\beta_k|\  \Big|\ \exists (a_k)_{k\geq1}\subset\mathbb B_i \colon f_1 =  \alpha_0 + \sum_{k=1}^{\infty} \varphi_{a_k}\alpha_k,\  f_2 = \beta_0 +\sum_{k=1}^n\varphi_{a_k}\beta_k \right\}.
\end{align}
On the other hand, an argument like the one used for the triangle inequality shows that we can vary $\alpha_k$ and $\beta_k$ independently, which allows us to split the infimum. So we get
\begin{multline*}
\|f\|_{\besov_1,i} \leq \inf\left\{\sum_{k=0}^{\infty}|\alpha_k|\  \Big|\ \exists (b_k)_{k\geq1}\subset\mathbb B_i \colon f_1 =  \alpha_0 + \sum_{k=1}^{\infty} \varphi_{b_k}\alpha_k \right\} \\
+ \inf\left\{\sum_{k=0}^{\infty} |\beta_k|\  \Big|\ \exists (c_k)_{k\geq1}\subset\mathbb B_i \colon f_2 = \beta_0 + \sum_{k=1}^n\varphi_{c_k}\beta_k \right\} \leq \|f_1\|_{\besov_{1, \mathbb{C}}} + \|f_2\|_{\besov_{1, \mathbb{C}}},
\end{multline*}
where $\|\cdot\|_{\besov_{1, \mathbb{C}}}$ is norm on the complex Besov space $\besov_{1, \mathbb{C}}$ on $\mathbb B_i$, which is defined analogously to \eqref{norm_Besov1}. \\
On the other hand it is clear that $\|f_k\|_{\besov_{1, \mathbb{C}}} \leq \|f\|_{\besov_1,i}$ for $k=1,2$. Thus, putting all these inequalities together, we obtain
$$
\|f_k\|_{\besov_{1, \mathbb{C}}} \leq \|f\|_{\besov_1,i} \leq \|f_1\|_{\besov_{1, \mathbb{C}}} + \|f_2\|_{\besov_{1, \mathbb{C}}}.
$$
Therefore, for any Cauchy sequence $(f_n)_{n\geq 0}$ in $\besov_{1,i}$, the sequences of the component functions $(f_{n,1})_{n\geq 1}$ and $(f_{n,2})_{n\geq1}$  are Cauchy sequences in $\besov_{1,\mathbb C}$. But as the complex Besov space $\besov_{1,\mathbb C}$ is complete with the norm $\|\cdot\|_{\besov_{1,\mathbb C}}$, see \cite{Arazy}, there exist limit functions $f_1$ and $f_2$ of  $(f_{n,1})_{n\geq1}$ and $(f_{n,2})_{n\geq 1}$ in $\besov_{1,\mathbb C}$ and the above inequality implies that the function $f=f_1 + f_2j$  is the limit of $(f_n)_{n\geq 1}$ in $\mathbb \besov_{1,i}$. Thus,   $\mathbb \besov_{1,i}$ is complete.

\end{proof}
\begin{remark}{\rm
If we work modulo constants and modulo linear terms, then we obtain that,
for $i,j\in\mathbb S^2$, the norms on $\besov_{1,i}$ and $\besov_{1,j}$  are equivalent.
}
\end{remark}
\begin{lemma}
Let $i\in\mathbb S^2$. Then, for $f\in\besov_{1,i}$, we have
$$\frac{1}{16\pi}\int_0^1\int_{0}^{2\pi}|\partial_{x_0}^2f(re^{i\theta})|\,d\theta\,dr \leq \|f - f(0) - z\partial_{x_0}f(0)\|_{\besov_{1,i}} \leq\frac{1}{\pi}\int_0^1\int_{0}^{2\pi}|\partial_{x_0}^2f(re^{i\theta})|\,d\theta\,dr.$$
\end{lemma}
\begin{proof}
Let $j\in\mathbb S^2$ with $i\perp j$. Then there exist $f_1$ and $f_2$ in the complex Besov space $\besov_{1,\mathbb C}$ on $\mathbb B_i$ such that $f = f_1 + f_2j$. Moreover, if $T_a$ denotes again the complex M\"obius transformation  on $\mathbb B_i$ associated with $a$, then $f = \gamma_0 + \sum_{k=1}^{\infty}T_{a_k}\gamma_k$ if and only if $f_k = \gamma_{0,l} + \sum_{k=1}^{\infty} T_{a_k} \gamma_{k,l}$, l=1,2, with $\gamma_k = \gamma_{k,1} + \gamma_{k,2}j$. Thus, we have
$$\|f_l\|_{\besov_{1,\mathbb C}}\leq \| f\|_{\besov_{1,i}}\leq \| f_1\|_{\besov_{1,\mathbb C}} + \| f_2\|_{\besov_{1,\mathbb C}},\qquad\text{for }l=1,2.$$
Because of Theorem 8 in \cite{Arazy}, we have
$$
\frac{1}{8\pi}\int_0^1\int_{0}^{2\pi}|g''(re^{i\theta})|\,d\theta dr \leq \| g - g'(0)-g''(0)z\|_{\besov_{1,\mathbb C}} \leq \frac{1}{\pi}\int_0^1\int_{0}^{2\pi}|g''(re^{i\theta})|\,d\theta dr
$$
for any function $g\in\besov_{1,\mathbb C}$. Thus, on one hand, we have
$$
\frac{1}{16\pi}\int_0^1\int_{0}^{2\pi}|\partial_{x_0}^2f(re^{i\theta})|\,d\theta\,dr \leq \frac{1}{16\pi}\int_0^1\int_{0}^{2\pi}|f_1''(re^{i\theta})|\,d\theta\,dr +\frac{1}{16\pi}\int_0^1\int_{0}^{2\pi}|f_2''(re^{i\theta})|\,d\theta\,dr
$$
$$
\leq \frac{1}{2} \|f_1 - f_1(0) - f'_1(0)z\|_{\besov_{1,\mathbb C}} + \frac{1}{2}  \|f_2 - f_2(0) - f'_2(0)z\|_{\besov_{1,\mathbb C}} \leq  \|f - f(0) - z\partial_{x_0}f(0) \|_{\besov_{1,i}}.
$$
On the other hand, we have
$$
\|f - f(0) - z\partial_{x_0}f(0) \|_{\besov_{1,i}} \leq  \|f_1 - f_1(0) - f_1'(0)z\|_{\besov_{1,\mathbb C}} +  \|f_2 - f_2(0) - f_2'(0)z\|_{\besov_{1,\mathbb C}}
$$
$$
\leq \frac{1}{2\pi}\int_0^1\int_{0}^{2\pi}|f'_1(re^{i\theta})|\,d\theta\,dr + \frac{1}{2\pi}\int_0^1\int_{0}^{2\pi}|f_2'(re^{i\theta})|\,d\theta\,dr \leq \frac{1}{\pi}\int_0^1\int_{0}^{2\pi}|\partial_{x_0}^2f(re^{i\theta})|\,d\theta\,dr.
$$

\end{proof}

\begin{corollary}
Let $i,j\in\mathbb S^2$ and let $f\in\besov_1$. Then
$$
\| f - f(0) - z\partial_{x_0}f(0)\|_{\besov_{1,j}}\leq 32\| f - f(0) - z\partial_{x_0}f(0)\|_{\besov_{1,i}}.
$$
Moreover, we can define a norm on   $\besov_{1}$ by
$$ \| f\|_{\besov_1} = \sup_{i\in\mathbb S^2}\|f\|_{\besov_{1,i}}.$$
\end{corollary}
\begin{proof}
By the representation formula and the previous lemma, we obtain
$$\| f - f(0) - z\partial_{x_0}f(0)\|_{\besov_{1,j}}\leq \frac{1}{\pi}\int_0^1\int_{0}^{2\pi}|\partial_{x_0}^2f(re^{j\theta})|\,d\theta\,dr $$
$$\leq \frac{1}{2\pi}\int_0^1\int_{0}^{2\pi}|(1-ji)\partial_{x_0}^2f(re^{i\theta})|\,d\theta\,dr + \frac{1}{2\pi}\int_0^1\int_{0}^{2\pi}|(1+ji)\partial_{x_0}^2f(re^{-i\theta})|\,d\theta\,dr $$
$$ = \frac{2}{\pi}\int_0^1\int_{0}^{2\pi}|\partial_{x_0}^2f(re^{i\theta})|\,d\theta\,dr \leq 32 \| f - f(0) - z\partial_{x_0}f(0)\|_{\besov_{1,i}}.$$
Moreover, because of this inequality, the function $\|f\|_{\besov_1} = \sup_{i\in\mathbb S^2}\|f\|_{\besov_{p,i}}$ is well defined. It is trivial to check that it actually is a norm.
\end{proof}
\begin{remark}
{\rm
Note that the local norm on a slice $\|f\|_{\besov_{1,i}}$ is invariant under M\"{o}bius transformations $T_a$ with $a\in\mathbb B_i$, but that
$\|f \|_{\besov_{1}}$ is not necessarily  invariant under M\"{o}bius transformations.
}
\end{remark}

In usual complex analysis, the reproducing kernel of the holomorphic weighted Bergman space for $p=2$ with weight
$(\alpha +1)(1-|z|^2)^{\alpha}$ is $$K_{\alpha}^{\mathbb C}(z,w) = \frac{1}{(1-z\bar w)^{2+\alpha}},$$
see for instance Corollary 4.20 in \cite{kehe}. This motivates the following definition.
\begin{definition}
Let $\alpha>-1$. We denote by $K_{\alpha}(\cdot,\cdot)$ the {\em $\alpha$-weighted slice regular Bergman kernel }  defined on $\mathbb B\times\mathbb B$. This kernel is defined by
$$
K_{\alpha}(\cdot,w) = P_{i_{z}}\left[ \frac{1}{(1-z\bar w)^{2+\alpha}}\right],
$$
where  $\displaystyle i_z = \frac{\vec{z}}{\|\vec{z}\|
}$.
 \end{definition}
\begin{remark}{\rm
Observe that the function $K_{\alpha}(\cdot,w)$ is the left slice regular extension of $K_{\alpha}^{\mathbb C}$ in the $z$ variable.
In alternative, it can be computed as the right slice regular extension in the variable $\bar w$
}
\end{remark}

As customary, by $L^p (\mathbb B_i, d\lambda_i, \mathbb H)$, we denote the set of functions $h\colon\mathbb B_i \to \mathbb H$ such that $$    \int_{\mathbb B_i} |h(w)|^p d\lambda_i(w) < \infty. $$
Note that for $j\in \mathbb S^2$ with $j \perp i$  and $h= h_1+h_2 j$ with $h_1,h_2 : \mathbb B_i \to \mathbb C(i)$, one can see directly that $h\in  L^p (\mathbb B_i, d\lambda_i, \mathbb H)$ if and only if
$h_1,h_2 \in L^p (\mathbb B_i, d\lambda_i,  \mathbb C(i))$.

The Bergman projection has been introduced in \cite{CGS}. Here we extend this notion to the more general setting of weighted Bergman spaces.
\begin{definition}
For $h\in  L^p(\mathbb B_i, d\lambda_i ,  \mathbb H)$ and $p>0$, we define the {\em Bergman type projection}
$$
\mathbf{K}_{\alpha,i}[h](q) =\int_{\mathbb B_i}   K(q,w) h(w) dA_{\alpha, i}(w) , \quad q\in \mathbb B.
$$
\end{definition}

As in the complex space, we have the following reproducing property on Bergman spaces.
\begin{lemma}
Let $i\in\mathbb S^2$ and $\alpha > -1$. If $f\in\berg_{\alpha,i}^{1}$ (see Definition \ref{4.1}) then
\begin{equation}\label{repprop} f(q) = \int_{\mathbb B_i}K_{\alpha}(q,w)f(w)\,dA_{\alpha,i}(w) \quad \text{for all }q\in\mathbb B.\end{equation}
\end{lemma}
\begin{proof}
Let $j\in\mathbb S^2$ with $i\perp j$ and let $f_1,f_2\colon \mathbb B_i\to\mathbb C(i)$ be holomorphic functions such that $ Q_i[f] = f_1 + f_2 j$. Because of Remark \ref{bergmanf1f2}, we know that $f_1,f_2$ are in the complex Bergman space $\berg_{\alpha,\mathbb C}^1$. Moreover, for $z,w\in\mathbb B_i$, we have $K_{\alpha}(z,w) = K_{\alpha}^{\mathbb C}(z,w)$ and so
\begin{align*} \mathbf{K}_{\alpha,i}[f](z) &=  \int_{\mathbb B_i}K_{\alpha}^{\mathbb C}(z,w)f(w)\,dA_{\alpha,i}(w) \\
&
=  \int_{\mathbb B_i}K_{\alpha}^{\mathbb C}(z,w)f_1(w)\,dA_{\alpha,i}(w) +  \int_{\mathbb B_i}K_{\alpha}^{\mathbb C}(z,w)f_2(w)\,dA_{\alpha,i}(w) j
\\
&
= f_1(z) + f_2(z)j = f(z),
\end{align*}
where we use the reproducing property of $K_{\alpha}^{\mathbb C}(\cdot,\cdot)$ on the complex Bergman space $\berg_{\alpha,\mathbb C}^p$, see Proposition 4.23 in \cite{kehe}. Thus
$$
f(z) =  \int_{\mathbb B_i}K_{\alpha}^{\mathbb C}(z,w)f(w)\,dA_{\alpha,i}(w)
$$
holds for any $z\in\mathbb B_i$ and so by taking the extension with respect to the variable $z$ we have
$$
f(q)=P_i\left[ \int_{\mathbb B_i}K_{\alpha}^{\mathbb C}(z,w)f(w)\,dA_{\alpha,i}(w)\right]$$
$$
=  \int_{\mathbb B_i}P_i[K_{\alpha}^{\mathbb C}(z,w)]f(w)\,dA_{\alpha,i}(w)
=  \int_{\mathbb B_i}K_{\alpha}(q,w)f(w)\,dA_{\alpha,i}(w).
$$
\end{proof}

\begin{proposition}
Let $ \alpha >-1$  and   $p\geq 1$ and let $f\in\mathcal{SR}(\mathbb B)$. Then $f\in \besov_{p}$ if and only if $ f \in  \mathbf{K}_{\alpha,i}  L^p(\mathbb B_i, d\lambda_i, \mathbb H)$ for  $ \ i\in\mathbb S^2$.
\end{proposition}
\begin{proof}
Since $f\in \besov_{p}$ then $f\in \besov_{p,i}$ for $i\in\mathbb S^2$.
Let $j \in\mathbb S^2$ be such that $j\perp i$. Then for $f\in \mathcal{SR}(\mathbb B)$ there exist holomorphic functions $f_1,f_2\colon\mathbb B_i\to \mathbb C(i)$ such that $ Q_i[f]= f_1+f_2 j$. By Remark \ref{besovf1f2}, we have  $f\in \besov_{p,i}$ if and only if $f_1$ and $f_2 $ belong to the complex Besov space $\besov_{p, \mathbb{C}}(\mathbb B_i)$, where $B_i$ is identified with $\mathbb{D} \subset \mathbb{C}_(i)$.

From Theorem 5.20 in \cite{kehe}, it follows that $f_1,f_2 \in\besov_{p,\mathbb{C}}(\mathbb B_i)$ if and only if there exist functions
$g_1,g_2\in L^p(\mathbb B_i, d\lambda_i, \mathbb C(i))$ such that $f_k = \mathbf{K}_{\alpha}[g_k]$ for $k=1,2$ where $\mathbf{K}_{\alpha}$ is the corresponding complex operator, that is
$$\mathbf{K}_{\alpha}[g_i] (z) =\int_{\mathbb B_i}   K_{\alpha}^{\mathbb C}(z,w) g_i(w) dA_{\alpha, i}(w)\quad \text{for }z\in\mathbb B_i.$$
Note that $\mathbf{K}_{\alpha,i}[h] = P_i \circ \mathbf{K}_{\alpha}[h]$ if $h$ has only values in $\mathbb{C}(i)$. So we have
$$f = P_i[f_1 + f_2j] = P_i\left[\mathbf{K}_{\alpha}[g_1]\right] + P_i\left[\mathbf{K}_{\alpha}[g_2 ]\right]j = \mathbf{K}_{\alpha,i}[g_1] + \mathbf{K}_{\alpha,i}[g_2]j = \mathbf{K}_{\alpha,i}[g_1 + g_2j].
$$
Thus, as $g = g_1 + g_2j$ is in $L^p(\mathbb B_i, d\lambda_i, \mathbb H)$ if and only if the components $g_1, g_2$ are in $L^p(\mathbb B_i, d\lambda_i, \mathbb C(i))$, the statement is true.

\end{proof}
\noindent
Therefore  we have $\besov_p \cong  \mathbf{K}_{\alpha ,i}  L^p(\mathbb B_i, d\lambda_i, \mathbb H)$ as topological vector spaces, if $\mathbf{K}_{\alpha,i}  L^p(\mathbb B_i, d\lambda_i, \mathbb H)$ is given with the quotient norm.

\begin{proposition} Let  $p>1$ and $ \alpha>-1$ and let $f\in \mathcal{SR}(\mathbb B) $. Then $f\in \besov_{p}$ if and only if
\begin{equation}\label{intintH} \int_{\mathbb B_i}     \int_{\mathbb B_i}  \frac{ | f(z) -  f(w) |^p    }{ | 1- z\bar w    |^{2(2+\alpha)}  } \,dA_{\alpha, i}(z)\,dA_{\alpha, i}(w) < +\infty , \end{equation}
for $i\in\mathbb S^2$.
\end{proposition}
\begin{proof} Let $i,j\in\mathbb S^2$ such that $i\perp j$ and let $f_1,f_2\colon\mathbb B_i\to\mathbb C(i)$ be holomorphic functions such that $ Q_i[f] = f_1 + f_2j$. Then, for $z,w\in\mathbb{B}_i$, we have
$$|  f_l(z) -  f_l(w)     |^p   \leq |  f(z) -  f(w)     |^p\ \leq 2^{p-1} \left(   |  f_1(z) -  f_1(w)     |^p + |  f_2(z) -  f_2(w)     |^p   \right),$$ for $ l=1,2$. Thus, the condition \eqref{intintH} is satisfied if and only if
$$\int_{\mathbb B_i}     \int_{\mathbb B_i}  \frac{ | f_l(z) -  f_l(w) |^p    }{ | 1- z\bar w    |^{2(2+\alpha)}  } dA_{\alpha, i}(z)dA_{\alpha, i}(w) < +\infty$$ for $l=1,2$. But because of Theorem 5.21 in \cite{kehe} this holds if and only if $f_1$ and $f_2$ belong to the complex Besov space, which is equivalent to $f\in\besov_{p,i}$ and so also equivalent to  $f\in\besov_{p}$.

\end{proof}

\begin{proposition}\label{Lpembedding}
Let $i \in \mathbb S^2$ and let $p\geq 1$,  $pt >1$, and $\alpha > -1$. Then the integral operator $\mathbf{T} = \mathbf{T}_{\alpha,t,i}$
$$  \mathbf{T}f(z) = (  1-   |z|^2 )^t  \int_{\mathbb B_i}\frac{  ( 1- |w|^2  )^{\alpha}     }{  (1- z \bar w)^{ 2+t+\alpha }      }  f(w)\, dA_i(w)   $$
is an embedding of $\besov_{p,i}$ into $L^p( \mathbb B_i, d\lambda_i, \mathbb H)$.
\end{proposition}
\begin{proof}Let $\besov_{p, \mathbb{C}}$ be the complex Besov space on  $\mathbb B_i$ identified with $\mathbb D$. By Theorem 5.22 in \cite{kehe}, the operator $\mathbf{T}$ is an embedding of $\besov_{p, \mathbb{C}}$ into $L^p(\mathbb B_i,d\lambda_i, \mathbb C(i))$, thus
$$\mathbf{T}\otimes \mathbf{T}\colon \besov_{p, \mathbb{C}}(\mathbb B_i)^2\to L^p(\mathbb B_i,d\lambda_i, \mathbb C(i))^2$$
 is also an embedding.

Now, let $j\in\mathbb S^2$ with $i\perp j$ and let us define the operators
$$
{Q}_{\besov_{p,i}}\colon\left\{\begin{array}{ccl}\besov_{p,i} &\to& \besov_{p,\mathbb C}(\mathbb B_i)^2\\ f&\mapsto& (f_1,f_2) \quad\text{ with } {Q}_i[f] = f_1+f_2j \end{array}\right.
$$
and
$$
 \mathcal{J}\colon\left\{\begin{array}{ccl}L^p(\mathbb B_i,d\lambda_i, \mathbb C(i))^2 &\to& L^p(\mathbb B_i,d\lambda_i, \mathbb H)\\ (f_1,f_2) &\mapsto& f_1+ f_2j\end{array}\right..
$$
For any $\mathbb H$-valued function $f  = f_1 + f_2j$ we have
$$|f(z)|^p\leq 2^{p-1}(|f_1(z)|^p + |f_2(z)|^p) \leq 2^{p} |f(z)|^p,$$
for any $z\in\mathbb B_i$. Thus, one easily obtains corresponding estimates for the norms for the considered Besov and $L^p$-spaces. Therefore, the operators ${Q}_{\besov_{p,i}}$ and $\mathcal J$ are embeddings too.

Finally, as $\mathbf{T}f = \mathbf{T}f_1 + \mathbf{T}f_2j$, we have $\mathbf{T} = \mathcal{J}\circ \mathbf{T}\otimes \mathbf{T}\circ {Q}_{\mathcal{B}_{p,i}}$, that is the following diagram commutes
$$
\begin{CD}
\besov_{p,i} @>\mathbf{T}>> L^p(\mathbb B_i, d\lambda_i, \mathbb H)\\
@V {Q}_{\besov_{p,i}}VV @A\mathcal J AA\\
\besov_{p,\mathbb C}(\mathbb B_i)^2 @>\mathbf{T}\otimes \mathbf{T}>> L^p(\mathbb B_i, d\lambda_i, \mathbb C(i))^2
\end{CD}.
$$
Thus, $\mathbf{T}$ is an embedding.

\end{proof}
We conclude this section observing that
operator $\mathbf{T}$  is important to study the duality theorems for Besov spaces that will be investigated elsewhere.

\section{Dirichlet space}\label{DSpaces}

In usual complex analysis, the Dirichlet space $\dirichlet_{\mathbb C}$ is defined as the set of analytic functions in on the unit disc $\mathbb D$ such that
\begin{equation}\label{gaussss}
 \int_{\mathbb D} |f'(z)|^2 d\Omega(z) < + \infty,
 \end{equation}
where $d\Omega$ is the differential of area in the complex plane, that is $d\Omega = dxdy$ if $z=x+iy$. This motivates the following definition.
\begin{definition}
The {\em slice regular Dirichlet space} $\dirichlet$  is defined as the quaternionic right linear space of slice regular functions $f$ on $\mathbb B$ such that
\begin{equation}\label{eq16}  \sup_{i\in\mathbb S^2}\int_{\mathbb B_i}|   \partial_{x_0} f(z) |^2 d\Omega_i(z) < \infty  \end{equation}
where $d\Omega_i(z)$ is the differential of area in the plane $\mathbb C(i)$.
\end{definition}
\begin{remark}{\rm
Let $i,j\in\mathbb S^2$. Then from the Representation Formula one obtains
$$  |\partial_{x_0} f(x +j y)  |^2\leq 4 \left[     |\partial_{x_0} f(x +i y)  |^2   +  |\partial_{x_0} f(x  - i y)  |^2 \right]$$
and by integration one gets
$$ \int_{\mathbb B_j}  |\partial_{x_0} f(w)  |^2    d\Omega_j(w)   \leq 8 \int_{\mathbb B_i}     |\partial_{x_0} f(z)  |^2   d\Omega_i(z). $$
Thus, if (\ref{eq16}) is finite for some $i\in \mathbb{S}^2$ then it is finite for all $j\in \mathbb{S}^2$.
}
\end{remark}

\begin{remark}\label{dirichletf1f2}
{\rm Let $j\in \mathbb S^2$ with  $j\perp i$ and let $f_1,f_2\colon\mathbb B_i\to\mathbb C(i)$ be holomorphic functions such that $ Q_i[f]= f_1+ f_2j$. The identity $$|\partial_{x_0} f (z)|^2 = |f_1'(z)|^2 + |f_2'(z) |^2   \quad \text{for } z\in\mathbb B_i $$
implies that $f\in\dirichlet$ if and only if $f_1$ and $f_2$ belong to the usual complex Dirichlet space $\dirichlet_{\mathbb C}$.
}
\end{remark}

\begin{proposition}
Let $f\in \mathcal{SR}(\mathbb B)$ and let $a_n\in\mathbb H$ for $ n\geq 0$ such that $\displaystyle f(q)= \sum_{n=0}^\infty q^n a_n$.  If $f\in\dirichlet$ then $$\frac{1}{\pi}\int_{\mathbb B_i} |\partial_{x_0}f(z)|^2 d\Omega_i(z) = \sum_{n=1}^\infty n|a_n|^2 $$
for any $i\in\mathbb S^2$.
\end{proposition}
\begin{proof}
Let $i,j\in \mathbb S^2$ with  $j\perp i$. So it is
\begin{enumerate}
\item $ Q_i[f] = f_1+ f_2j$, where $f_1,f_2\in Hol(\mathbb B_i)$ and
\item $a_n= a_{1,n} + a_{2,n} j$, where $a_{1,n}, a_{2,n} \in \mathbb C(i)$ for $n\geq 0$.
 \end{enumerate}
Then we have $ f_k(q) = \sum_{n=0}^{\infty}q^na_{k,n}$ for $k=1,2$  and from Remark \ref{dirichletf1f2} we have that $f_1$ and $f_2$  belong to the complex Dirichlet space $\dirichlet_{\mathbb C}$. Thus, from Section 4 of \cite{danikas}, we have that
$$  \frac{1}{\pi}\int_{\mathbb B_i}|f_k'(z)|^2 d\Omega_i(z)= \sum_{n=1}^{\infty}  n|a_{k,n}|^2 , \quad \textrm{ for  } \ k=1,2$$
and so
\begin{align*}
\frac{1}{\pi}\int_{\mathbb B_i}|\partial_{x_0} f (z)|^2 d\Omega_i(z)
&=  \frac{1}{\pi}\int_{\mathbb B_i}|  f'_1(z)|^2 d\Omega_i(z) +  \frac{1}{\pi}\int_{\mathbb B_i}|  f'_2(z)|^2 d\Omega_i(z)
\\
&
 =   \sum_{n=1}^{\infty}  n|a_{1,n}|^2  +  \sum_{n=1}^{\infty}  n|a_{2,n}|^2
 \\
 &
 =   \sum_{n=1}^{\infty}  n|a_n|^2.
 \end{align*}
\end{proof}

\begin{definition}
On the slice regular Dirichlet space we define a norm by
$$\|f\|_{\dirichlet} = \left(|f(0)|^2 + \sup_{i\in\mathbb S^2}\int_{\mathbb B_i}|\partial_{x_0} f(z)|d\Omega_i(z)\right)^{\frac{1}{2}}.$$
\end{definition}

\begin{remark}
{\rm
Let $i,j\in\mathbb S^2$ with $i\perp j$ and let $f\in\dirichlet$. Furthermore, let  $f_1$ and $ f_2 $ be functions in the complex Dirichlet space $\dirichlet_{\mathbb C}$ on $\mathbb B_i$ such that $ Q_i[f]= f_1+ f_2j$. Then the equalities
$$ |f(0)|^2 =     |f_1(0)|^2   +  |f_2(0)|^2$$
and
\begin{equation}\label{Dnormf1f2}  \int_{\mathbb B_i } |  \partial_{x_0} f(z)|^2  d_z\Omega =  \int_{\mathbb B_i } |   f'_1(z)|^2  d_z\Omega  +    \int_{\mathbb B_i } |  f'_2(z)|^2  d_z\Omega   \end{equation}
imply
$$\|f\|^2_{\dirichlet} =   |f_1(0)|^2 + \int_{\mathbb B_i } | f_1'(z)|^2  d\Omega_i(z)  +    |f_2(0)|^2   +   \int_{\mathbb B_i } |   f_2'(z)|^2  d\Omega_i(z)       = \|f_1\|^2_{\dirichlet_{\mathbb C}}   + \|f_2\|^2_{\dirichlet_{\mathbb C}}  ,   $$
where  $ \|\cdot \|^2_{\dirichlet_{\mathbb C}} $ denotes the norm on the complex Dirichlet space $\dirichlet_{\mathbb C}$ on $\mathbb B_i$.
}
\end{remark}

\begin{proposition}
The function  space $( \dirichlet ,     \|\cdot\|_{\dirichlet} ) $ is a complete normed space.\end{proposition}
\begin{proof} The relation \eqref{Dnormf1f2} implies that $\|\cdot\|_{\dirichlet}$  is actually a norm.
Now let $i,j\in\mathbb S$ with $i\perp j$.
Moreover, let  $(f_n)_{n\geq0}$ be a Cauchy sequence  in $\dirichlet$ and for $n\in\mathbb N$ let $f_{n,1}$ and $f_{n,2}$ be two functions in the complex Dirichlet space $\dirichlet_{\mathbb C}$ on $\mathbb B_i$ such that $ Q_i[f_n] = f_{n,1} + f_{n,2}j$. Because of \eqref{Dnormf1f2}, the sequences $(f_{n,})_{n\geq 1}$ and $(f_{n,2})_{n\geq 1} $ are Cauchy sequences in $\dirichlet_{\mathbb C}$.
As showed in Section 4 of \cite{danikas} the complex Dirichlet space is complete and so there exist functions $f_1$ and $f_2$ such that  $f_{n,1}\to f_1$ and $f_{n,2}\to f_2$ in $\dirichlet_{\mathbb C}$ as $n\to\infty$. Now set $f = P_i[f_1 + f_2j]$. Then $f\in\dirichlet$ and
$$\| f - f_n\|_{\dirichlet}^2 = \|f_1 - f_{n,1}\|^2_{\dirichlet_{\mathbb C}} - \|f_2 - f_{n,2}\|^2_{\dirichlet_{\mathbb C}} \to 0.$$
Thus, the slice regular Dirichlet space is complete.

\end{proof}

Moreover, the slice regular Dirichlet space  has the structure of a quaternionic Hilbert space.
\begin{definition}
Let $i\in\mathbb S^2$. For $f,g\in \dirichlet $ we define their inner product as
$$ \langle f,g\rangle_{ \dirichlet} := \bar f(0) g(0) + \sup_{i\in\mathbb S^2}\int_{\mathbb B_i} \overline{\partial_{x_0} f(z)}\partial_{x_0} g(z) d\Omega_i(z).$$
\end{definition}
Because of the Cauchy Schwarz inequality, this inner product is well defined.
\begin{proposition}
The function $ \langle\cdot ,\cdot \rangle_{ \dirichlet} $ is a quaternionic right linear inner product on $\dirichlet$. Precisely, for all $f,g,h\in\dirichlet$ and all $\lambda\in\mathbb H$, we have
\begin{enumerate}[(i)]
\item right linearity: $\langle f, g\lambda + h\rangle_{\dirichlet} = \langle f, g\rangle_{\dirichlet}\lambda +\langle f,h\rangle_{\dirichlet}$
\item  quaternionic hermiticity: $\langle g, f\rangle_{\dirichlet} = \overline{\langle f, g\rangle_{\dirichlet}}$
\item positivity: $\langle f, f\rangle_{\dirichlet} \geq 0$ and $\langle f, f\rangle_{\dirichlet} = 0$ if and only if $f=0$.
\end{enumerate}
\end{proposition}

\begin{proposition}
The space $(\dirichlet, \langle\cdot ,\cdot \rangle_{ \dirichlet} )$ is a quaternionic right Hilbert space.
\end{proposition}
\begin{proof}
The previous proposition shows that $\langle\cdot ,\cdot \rangle_{ \dirichlet}$  is a quaternionic right linear inner product. Furthermore, the induced norm $ \sqrt{\langle f ,f \rangle_{ \dirichlet} }$ coincides with $\|f \|_{\dirichlet}$, so the space is also complete.
\end{proof}

\end{document}